\documentclass[12pt]{article}

\newtheorem{theorem}{Theorem}
\newtheorem{definition}{Definition}

\newtheorem{lemma}{Lemma}
\newtheorem{corollary}{Corollary}
\newtheorem{notation}{Notation}

\newcommand{\al}{\alpha}
\newcommand{\ale}{{\alpha_1}}

\newcommand{\alz}{\alpha_2}

\newcommand{\als}{\alpha'}
\newcommand{\bes}{\beta'}
\newcommand{\bez}{\beta_2}
\newcommand{\bet}{\tilde{\beta}}
\newcommand{\alt}{\tilde{\alpha}}

\newcommand{\de}{\delta}
\newcommand{\be}{\beta}
\newcommand{\bee}{{\beta_1}}
\newcommand{\ga}{\gamma}
\newcommand{\vt}{\vartheta}
\newcommand{\vte}{\vartheta_1}
\newcommand{\vti}{\vartheta_i}

\newcommand{\di}{d_i}
\newcommand{\dipe}{d_{i+1}}
\newcommand{\vtn}{\vartheta_0}
\newcommand{\vtz}{\vartheta_2}

\newcommand{\oti}{\overline{\theta}_i}
\newcommand{\ot}{\overline{\theta}}

\newcommand{\om}{\omega}
\newcommand{\Om}{\Omega}
\newcommand{\und}{\;\&\;}
\newcommand{\imp}{\Rightarrow}

\newcommand{\ol}{\overline}
\newcommand{\oa}{\overline{\al}}
\newcommand{\ob}{\overline{\be}}

\newcommand{\wpo}{\mathbb{wpo}}

\renewcommand{\S}{\mathbb{S}}
\newcommand{\N}{\mathbb{N}}

\newcommand{\rca}{\mathsf{RCA}_0}

\newcommand{\aca}{\mathsf{ACA}_0}

\newcommand{\pica}{\Pi^1_1\mathsf{\text{-}CA}_0}

\usepackage{amsmath}
\usepackage{amssymb}
\usepackage{bbold}
\usepackage{enumitem}

\newenvironment{proof}{\textit{Proof.}}{\hfill$\square$}

\title{Ordinal notation systems corresponding to Friedman's linearized well-partial-orders with gap-condition}

\author{Michael Rathjen\thanks{rathjen@maths.leeds.ac.uk, Department of Pure Mathematics, University of Leeds, Leeds LS2 9JT, England},
Jeroen Van der Meeren\thanks{jeroen.vandermeeren@ugent.be, Department of Mathematics, Ghent University, Krijgs\-laan 281, 9000 Gent, Belgium},
Andreas Weiermann\thanks{andreas.weiermann@ugent.be, Department of Mathematics, Ghent University, Krijgs\-laan 281, 9000 Gent, Belgium}
}

\begin{document}
\maketitle
\begin{abstract}
In this article we investigate whether the addition-free theta functions
form a canonical notation system for the linear versions of Friedman's
well-partial-orders with the so-called gap-condition over a finite set of labels. Rather surprisingly, we can show
this is the case for two labels, but not for more than two labels. To this end, we determine the order type of the
notation systems for addition-free theta functions in terms of ordinals
less than $\varepsilon_0$.
We further show that the maximal order type of the Friedman ordering
can be obtained by a certain ordinal notation system which
is based on specific binary theta functions.
\end{abstract}

\section{Introduction}
A major theme in proof theory is to provide natural independence results for formal systems for reasoning about mathematics.
The most prominent system in this respect is first order Peano arithmetic, or almost equivalently its second order version $\aca$.
Providing natural independence results for stronger systems turned out to be rather difficult. The strongest system considered in reverse mathematics \cite{SOSOA} is $\pica$ which formalizes full $\Pi^1_1$-comprehension (with paramters) over $\rca$.
Buchholz \cite{BuchholzHydra} provided a natural hydra game for $\pica$ but this follows closely a path which is delineated by the classification of the provably recursive functions in terms of a corresponding Hardy hierarchy. Harvey Friedman \cite{simpsonfinitetrees} obtained a spectacular independence result for $\pica$ by considering well-quasi-orders on labeled trees on which he imposed a so-called gap-condition. It is still open to classify the strength of Friedman's assertion for the case that the set of labels consists of $n$ elements where $n$ is fixed from the outside. Nowadays it is known that the proof-theoretic strength of a well-quasi-order-principle can be measured in terms of the maximal order type of the well-quasi-order under consideration.
The maximal order type for the Friedman ordering is known for $n=1$ by results of Schmidt and Friedman. Recently the case $n=2$ has been settled and the case for $n\geq 3$ seems to be possible to obtain.
It turned out that the maximal order type for $n=2$ can be expressed using higher collapsing functions $\vartheta_0$ and $\vartheta_1.$

As a preliminary step in classifying the general case it seems natural to classify the situation where trees are replaced by sequences over a finite set of cardinality $n$.
The hope is that the simpler case indicates how to deal with the general case of trees.
Investigations on finite sequences with respect to the Friedman ordering have been undertaken by Sch\"utte and Simpson \cite{schuttesimpson}.
They showed how the Friedman ordering can be reduced to suitably nested versions of the Higman ordering \cite{Higman}.
Moreover they considered the corresponding Buchholz-style ordinal notation system in which the addition function has been dropped.
Curiously this lead to an ordinal notation system which in the limit (for unbounded $n$) reached $\epsilon_0$.
It is quite natural to consider finite sequences as iterated applications of unary functions and it is quite natural to ask whether the ordinal notation system which is based on $n$ collapsing functions (which in \cite{schuttesimpson} are denoted by $\pi_0,\ldots,\pi_{n}$) generates the maximal order type for the Friedman ordering $S_n$ for sequences over a set with $n$ elements. But it turns out that this is not the case: to produce the maximal order type for $S_n$ one needs the functions $\pi_0,\ldots,\pi_{2n}$.
It is known that the so called theta functions $\theta_i$ grow more quickly than the functions $\pi_i$ and it is natural whether their addition-free analogues $\vartheta_0,\ldots,\vartheta_{n}$ generate the maximal order type of $S_n$. For $n=2$ this turned out to be true and so one would expect that this would generalize to $n\geq 3$.
Quite surprisingly this is again not the case.
To obtain the maximal order type of $S_n$ one requires the functions $\vartheta_0,\ldots,\vartheta_{2\cdot n-3}$.

So the question remains whether $S_n$ can be realized by a suitable choice of unary functions. It turns out that  this, as we will show, is indeed possible using specific binary theta functions. However, with unary functions the question is still open.

In a sequel project, we intend to determine the relationship between other ordinal notation systems without addition (e.g. ordinal diagrams \cite{Takeuti}, Gordeev-style notation systems \cite{GordeevAML1989} and non-iterated $\vartheta$-functions \cite{prooftheoryofimpredicativesubsystemsofanalysis,WilkenSimultanetheta}) with the systems used in this article.




\section{Preliminaries}
\subsection{Well-partial-orders}

Well-partial-orders are the natural generalizations of well-orders. They have applications in computer science, commutative algebra and logic.

\begin{definition}
A \textbf{well-partial-order} (hereafter $\wpo$) is a partial order that is well-founded and does not admit infinite antichains. Hence, it is a partial order $(X,\leq_X)$ such that for every infinite sequence $(x_i)_{i<\omega}$ in $X$ there exist two indices $i<j$ such that $x_i \leq_X x_j$. If the ordering is clear from the context, we do not write the subscript $X$.
\end{definition}

$\wpo$'s appear everywhere in mathematics. For example, they are the main ingredients in Higman's theorem \cite{Higman}, Graph Minor theorem \cite{graphtheorydiestel}, Fra\"iss\'e's order type conjecture \cite{laver} and Kruskal's theorem \cite{Kruskal}. The latter is used in field of term rewriting systems.

\medskip

In this paper, we are interested in $\wpo$'s with the so-called gap-condition introduced in \cite{simpsonfinitetrees}. We are especially interested in the linearized version, which is already studied by Sch\"utte and Simpson \cite{schuttesimpson} (see subsection \ref{Wpo with gap-condition} for more information). With regard to these $\wpo$'s, we want to study ordinal notation systems which correspond to their maximal order types and maximal linear extensions.

\begin{definition} The \textbf{maximal order type} of the $\wpo$ $(X,\leq_X)$ is equal to $\sup\{\alpha$: $\leq_X\subseteq \preceq$, $\preceq$ is a well-order on $X$ and $otype(X,\preceq)=\alpha\}$. We denote this ordinal by $o(X,\leq_X)$ or by $o(X)$ if the ordering is obvious from the context.
\end{definition}

The following theorem by de Jongh and Parikh \cite{dejonghandparikh} shows that this supremum is actually a maximum.

\begin{theorem}[de Jongh and Parikh \cite{dejonghandparikh}]\label{dejonghandparikh1} Assume that $(X,\leq_X)$ is a $\wpo$. Then there exists a well-order $\preceq$ on $X$ which is an extension of $\leq_X$ such that $otype(X,\preceq) = o(X,\leq_X)$.
\end{theorem}

\begin{definition} Let $X$ be a $\wpo$. Every well-order $\preceq$ on $X$ that satisfies Theorem \ref{dejonghandparikh1} is called a \textbf{maximal linear extension}.
\end{definition}

The following definition and lemma are very useful.

\begin{definition}
A \textbf{quasi-embedding} $e$ from the partial order $(X,\leq_X)$ to the partial order $(Y,\leq_Y)$ is a mapping such that for all $x,x'\in X$, if $e(x)\leq_Y e(x')$, then $x\leq_X x'$ holds.
\end{definition}

\begin{lemma}\label{quasi-embedding}
Assume that $e$ is a quasi-embedding from the partial order $X$ to the partial order $Y$. If $Y$ is a $\wpo$, then $X$ is also a $\wpo$ and $o(X) \leq o(Y)$.
\end{lemma}

\begin{notation}
Let $\al$ be an ordinal. Define $\omega_0[\al]$ as $\al$ and $\omega_{n+1}[\al]$ as $\omega^{\omega_n[\al]}$. Write $\omega_n$ for the ordinal $\omega_n[1]$.
\end{notation}

\subsection{Well-partial-orders with the gap-condition}\label{Wpo with gap-condition}

In 1982, Harvey Friedman introduced a well-partial-order of finite rooted trees with labels in $\{0,\dots,n-1\}$ with a gap-embeddability relation on it. This was later published by Simpson in \cite{simpsonfinitetrees}. This $\wpo$ was very important, because it was one of the first natural examples of statements not provable in the strongest theory of the Big Five in Reverse Mathematics, $\pica$.

\begin{definition}
Let $\mathbb{T}_n$ be the set of finite rooted trees with labels in $\{0,\dots,n-1\}$. An element of $\mathbb{T}_n$ is of the form $(T,l)$, where $T$ is a finite rooted tree, which we see as a partial order on the set of nodes, and $l$ is a labeling function, a mapping from $T$ to the set $\{0,\dots,n-1\}$. Define $(T_1,l_1)\leq_{gap} (T_2,l_2)$ if there exists an injective order- and infimum-preserving mapping $f$ from $T_1$ to $T_2$ such that
\begin{enumerate}
\item $\forall \tau \in T_1$, we have $l_1(\tau) = l_2(f(\tau))$.
\item $\forall \tau  \in T_1$ and for all immediate successors $\tau' \in T_1$ of $\tau$, we have
that if $\overline{\tau} \in T_2$ is strictly between $f(\tau)$ and $f(\tau')$, then $l_2(\overline{\tau}) \geq l_2(f(\tau')) = l_1(\tau')$.
\end{enumerate}
\end{definition}

\begin{theorem}[Simpson/Friedman\cite{simpsonfinitetrees}]
For all $n$, $(\mathbb{T}_n,\leq_{gap})$ is a $\wpo$ and
$\pica \not\vdash \, \forall n< \omega$ `$(\mathbb{T}_n,\leq_{gap})$ is a $\wpo$'.
\end{theorem}

We are interested in the linearized versions of these $\wpo$'s, which have been studied extensively by Sch\"utte and Simpson \cite{schuttesimpson}. Before we give the definition of these linearized $\wpo$'s, we introduce the disjoint sum and cartesian product between $\wpo$'s and the Higman ordering.

\begin{definition} Let $X_0$ and $X_1$ be two $\wpo$'s. Define the \textbf{disjoint sum} $X_0 + X_1$ as the set $\{(x,0): x \in X_0\}\cup \{(y,1): y \in X_1\}$ with the following ordering:
\[(x,i) \leq (y,j) \Leftrightarrow i=j \text{ and } x \leq_{X_i} y.\]
For an arbitrary element $(x,i)$ in $X_0 + X_1$, we omit the second coordinate $i$ if it is clear from the context to which set the element $x$ belongs to.
Define the \textbf{cartesian product} $X_0 \times X_1$ as the set $\{(x,y): x  \in X_0, y \in X_1\}$ with the following ordering:
\[(x,y)\leq (x',y') \Leftrightarrow x\leq_{X_0} x' \text{ and } y\leq_{X_1} y'.\]
\end{definition}

\begin{definition}
Let $X^*$ be the \textbf{set of finite sequences} over the partial order $(X,\leq_X)$. Denote $x_0\dots x_{k-1} \leq_X^* y_0\dots y_{l-1}$ if there exists a strictly increasing function $f:\{0,\dots,k-1\} \to \{0,\dots,l-1\}$ such that for all $0 \leq i \leq k-1$, $x_i \leq_X y_{f(i)}$ holds. If the ordering on $X$ is clear from the context, we write $X^*$ instead of ($X^*,\leq_X^*)$.
\end{definition}

Hence, if we write $X^*$, we mean the set of of finite sequences over $X$ or the partial order $(X^*,\leq_X^*)$. The context will make clear what we mean.
Define $\S_n$ as $\{0,\dots,n-1\}^*$ and $\S$ as $\N^*$. $\S_n$ and $\S$ are either sets of finite sequences or partial orders.

\begin{theorem}[de Jongh-Parikh\cite{dejonghandparikh}, Schmidt\cite{dianaschmidt}]
\label{maximal order type sum product and higman}If $X_0$, $X_1$ and $X$ are $\wpo$'s, then $X_0 + X_1$, $X_0 \times X_1$ and $X^*$ are still $\wpo$'s, and
\begin{align*}
o(X_0 + X_1) &= o(X_0)\oplus o(X_1),\\
o(X_0 \times X_1) &= o(X_0) \otimes o(X_1),
\end{align*}
where $\oplus$ and $\otimes$ is the natural sum and product between ordinals, and
\[o(X^*) = \left\{
\begin{array}{ll}
\omega^{\omega^{o(X)-1}} & \text{if $o(X)$ is finite,}\\
\omega^{\omega^{o(X)+1}}& \text{if $o(X)= \varepsilon + n$, with $\varepsilon$ an epsilon number and $n< \omega$,}\\
\omega^{\omega^{o(X)}} & \text{otherwise.}
\end{array}
\right.\]
\end{theorem}

%


Now, we define the linearized versions of the gap-embeddability relation. 

\begin{definition}\label{definitieweakstrongnormalgapsequences}
In this context, let $\S_n$ be the set of the finite sequences over $\{0,\dots,n-1\}$. We say that $s =s_0\dots s_{k-1} \leq_{gap}^w s'_0\dots s'_{l-1} = s'$ if there exists a strictly increasing function $f:\{0,\dots,k-1\} \to \{0,\dots,l-1\}$ such that
\begin{enumerate}
\item for all $0 \leq i \leq k-1$, we have $s_i = s'_{f(i)}$,
\item for all $0 \leq i < k-1$ and all $j$ between $f(i)$ and $f(i+1)$, the inequality $s'_j \geq s'_{f(i+1)} = s_{i+1}$ holds.
\end{enumerate}
This ordering on $\S_n$ is called the \textbf{weak gap-embeddability relation}. The partial order $(\S_n,\leq_{gap}^w)$ is also denoted by $\S_n^{w}$.
The \textbf{strong gap-em\-bed\-da\-bility relation} fulfills the extra condition
\begin{enumerate}
\item[3.] for all $j<f(0)$, we have $s'_j \geq s'_{f(0)} = s_0$.
\end{enumerate}
This ordering on $\S_n$ is denoted by $\leq_{gap}^s$
We also write $\S_n^s$ for the partial order $(\S_n,\leq_{gap}^s)$.
\end{definition}

We now give an overview of the results in the article of Sch\"utte and Simpson \cite{schuttesimpson}.

\begin{theorem}[Sch\"utte-Simpson\cite{schuttesimpson}, Simpson/Friedman\cite{simpsonfinitetrees}]
For all $n$, $(\S_n,\leq^w_{gap})$ and $(\S_n,\leq_{gap}^s)$ are $\wpo$'s.
\end{theorem}

\begin{theorem}[Sch\"utte-Simpson\cite{schuttesimpson}]\ \\
$\aca \not\vdash \forall n < \omega \,$ `$(\S_n,\leq_{gap}^w)$ is a $\wpo$',\\
$\aca \not\vdash \forall n < \omega \,$ `$(\S_n,\leq_{gap}^s)$ is a $\wpo$'.
\end{theorem}

\begin{theorem}[Sch\"utte-Simpson\cite{schuttesimpson}]\ \\
For all $n$, $\aca \vdash$ `$(\S_n,\leq_{gap}^w)$ is a $\wpo$',\\
For all $n$, $\aca \vdash$ `$(\S_n,\leq_{gap}^s)$ is a $\wpo$'.
\end{theorem}

\begin{definition}
Denote the subset of $\S_n$ of elements $s_0 \dots s_{k}$ that fulfill the extra condition $s_0\leq i$ by $\S_n[i]$. Accordingly as in Definition \ref{definitieweakstrongnormalgapsequences}, $(\S_n[i], \leq^w_{gap})$, respectively $(\S_n[i],\leq^s_{gap})$, is denoted by $\S_n^w[i]$, respectively $\S_n^s[i]$.
\end{definition}

\begin{definition}
Take two partial orders $X_0$ and $X_1$. We say that $X_0$ and $X_1$ are order-isomorphic if there exists a bijective function $f$ such that $x\leq_{X_0} y \Leftrightarrow f(x)\leq_{X_1} f(y)$ for all $x$ and $y$ in $X_0$. We denote this by $X_0 \cong X_1$.
\end{definition}

If $X_0 \cong X_1$ and $X_0$ or $X_1$ is a $\wpo$, then the other one is also a $\wpo$ with the same maximal order type.

\medskip

The proofs by Sch\"utte and Simpson \cite{schuttesimpson} also yield results on the maximal order types of the sequences with the gap-embeddability relation. More specifically, they prove the next lemma (which is in Lemma 5.5 in \cite{schuttesimpson}). However, there is a small error in their proof, although we believe that this can actually be seen as a typo. For clarity reasons, the proof is given here.

\begin{theorem}[Sch\"utte-Simpson\cite{schuttesimpson}]\label{congruenceS_n+1strongandS_nstrong}\ \\
$\S_{n+1}^s \cong  \S_n^s \times (\S_n^s)^*$.
\end{theorem}
\begin{proof}
Assume $n\geq 0$. 
We define an order-preserving bijection $h_{n}$ from $\S_{n+1}^s$ to the partial order $\S_{n}^s \times (\S_{n}^s)^*$. Let $h_{n}(\varepsilon)$ be $(\varepsilon, ())$. Take an arbitrary element $s \in \S_{n+1}^s \backslash \{\varepsilon\}$. Then $s = s_0^+ 0 \dots 0 s_k^+$, with $s_i \in \S_{n}^s$ for all $i$ and $s_i^+$ is the result of replacing every number $j$ in $s_i$ by $j+1$. Define then $h_{n}(s)$ as $(s_0, (s_1,\dots,s_k))$. Note that for the sequence $s=0$, $k \geq 1$. In other words, $k$ represents the number of $0$'s occurring in $s$. It is easy to see that $h_n$ is a bijection.

\medskip

We know prove that $s < s'$ yields $h_n(s) < h_n(s')$ by induction on $lh(s) + lh(s')$. If $s$ or $s'$ is $\varepsilon$, then this is trivial. So assume $s = s_0^+ 0 \dots 0 s_k^+$ and $s' = {s'}_0^+ 0 \dots 0 {s'}_l^+$. If $k=0$, then $s<s'$ yields $l=0$ and $s_0^+ <{s'}_0^+$, or $l>0$ and $s_0^+ \leq {s'}_0^+$. In both cases, $h_n(s) < h_n(s')$. Assume $k>0$. Then $s< s'$ yields $l>0$, $s_0^+ \leq {s'}_0^+$ and $s_1^+ 0 \dots 0 s_k^+ \leq {s'}_j^+ 0 \dots 0 {s'}_l^+$ for a certain $j\geq 1$. From $s_1^+ 0 \dots 0 s_k^+ \leq {s'}_j^+ 0 \dots 0 {s'}_l^+$, one can prove as before (or by an additional induction argument on $k$) that $s_1^+ \leq {s'}_j^+$ and $s_2^+ 0 \dots 0 s_k^+ \leq {s'}_{j_2}^+ 0 \dots 0 {s'}_l^+$ for a certain $j_2 \geq j+1$. In the end, we have $s_0^+ \leq {s'}_0^+$ and $(s_1^+ , \dots,  s_k^+) \leq^* ({s'}_1^+ , \dots, {s'}_l^+)$. This yields $h_n(s) < h_n(s')$. The reverse direction $h_n(s) < h_n(s') \rightarrow s<s'$ can be proven in a similar way.

%
\end{proof}

\begin{corollary}\label{cor:equalordertypeS_nstrong}
$o(\S_{n+1}^s) = o(\S_n^s) \otimes o((\S_n^s)^*)$.
\end{corollary}

Hence, from the maximal order type of $\S_1^s$, which is the ordinal $\omega$, one can calculate the maximal order types of all $\S_n^s$.
Following the same template, one also has the following lemma.

\begin{lemma}\label{equalordertypeS_n+1andS_n}
$o(\S_{n+1}^w) = o(\S_{n+1}^w[0]) =o(\S_{n+1}^s[0])  =  o((\S_n^s)^*)$.
\end{lemma}
\begin{proof}
The equality $o(\S_{n+1}^s[0])  =  o((\S_n^s)^*)$ follows from the proof of Theorem \ref{congruenceS_n+1strongandS_nstrong}. $o(\S_{n+1}^w[0]) =o(\S_{n+1}^s[0])$ is trivial as they refer to the same ordering. To prove $o(\S_{n+1}^w) = o(\S_{n+1}^w[0])$, note that $ \S_{n+1}^w[0] \subseteq \S_{n+1}^w$, hence $o(\S^w_n[0]) \leq o(\S_n^w)$.
Furthermore, the mapping $e$ which plots $s_0\dots s_{k-1}$ to $0s_0\dots s_{k-1}$ is a quasi-embedding from $\S_n^w$ to $\S^w_n[0]$. Hence, $o(\S_n^w) \leq o(\S^w_n[0])$.

\end{proof}

\medskip

These results yield for example 
\[o(\S_2^w) = \omega^{\omega^{\omega}}.\] 
We are especially interested in substructures of the $\wpo$'s $\S_n^w$ and $\S_n^s$ such that their maximal order types are exactly equal to an $\omega$-tower, meaning it is of the form $\omega^{\omega^{\cdot^{\cdot^{\cdot^{\omega}}}}}$ (without any `$+1$'). Thereon, using Theorem \ref{congruenceS_n+1strongandS_nstrong}, Corollary \ref{cor:equalordertypeS_nstrong} and Lemma \ref{equalordertypeS_n+1andS_n}, we define the following.

\begin{definition}
Let $\overline{\S}_n$ be the subset of $\S_n$ which consists of all sequences $s_0\dots s_{k-1}$ in $\S_n$ such that for all $i<k-1$, $s_i - s_{i+1} \geq -1$. This means that if $s_i=j$, then $s_{i+1}$ is an element in $\{0,\dots,j+1\}$. For example $02\notin \overline{\S}_3$. Like in Definition \ref{definitieweakstrongnormalgapsequences}, we denote the subset of $\overline{\S}_n$ that fulfill the extra condition $s_0\leq i$ by $\overline{\S}_n[i]$. We denote $(\overline{\S}_n,\leq_{gap}^w)$ by $\overline{\S}_n^w$, $(\overline{\S}_n,\leq_{gap}^s)$ by $\overline{\S}_n^s$, $(\overline{\S}_n[i],\leq_{gap}^w)$ by $\overline{\S}^w_n[i]$ and $(\overline{\S}_n[i],\leq_{gap}^s)$ by $\overline{\S}^s_n[i]$.
\end{definition}

\begin{lemma}
$o(\overline{\S}_{n+1}^w) = o(\overline{\S}^w_{n+1}[0]) =o(\overline{\S}^s_{n+1}[0]) =  o((\overline{\S}^s_n[0])^*) = o((\overline{\S}^w_n[0])^*) = o((\overline{\S}^w_n)^*)$.
\end{lemma}
\begin{proof}
Similar as in Theorem \ref{congruenceS_n+1strongandS_nstrong}, Corollary \ref{cor:equalordertypeS_nstrong} and Lemma \ref{equalordertypeS_n+1andS_n}.
\end{proof}

\begin{corollary}
For all $n$, $o(\overline{\S}_n^w) = \omega_{2n-1}$.
\end{corollary}


\subsection{Ordinal notation systems}

In this subsection, we introduce several ordinal notation systems for ordinals smaller than $\varepsilon_0$. All of them do not use the addition operator.

\subsubsection{The Veblen hierarchy}
Assume that $(T,<)$ is a notation system with $otype(T) \in \varepsilon_0 \backslash \{0\}$. Define the representation system $\varphi_T 0 $ recursively as follows.

\begin{definition}
\begin{itemize}
\item $0 \in \varphi_T 0$,
\item if $\alpha \in \varphi_T 0$ and $t \in T$, then $\varphi_t \alpha \in \varphi_T 0$.
\end{itemize}
\end{definition}

Define on $\varphi_T 0$ the following total order.

\begin{definition} For $\alpha, \beta \in \varphi_T 0$, $\alpha < \beta$ is valid if
\begin{itemize}
\item $\alpha =0$ and $\beta \neq 0$,
\item $\alpha = \varphi_{t_1} \alpha'$, $\beta = \varphi_{t_2} \beta'$ and one of the following cases holds:
\begin{enumerate}
\item $t_1 < t_2$ and $\alpha ' < \beta$,
\item $t_1 = t_2$ and $\alpha' < \beta'$,
\item $t_1 > t_2$ and $\alpha \leq \beta'$.
\end{enumerate}
\end{itemize}
\end{definition}

\begin{theorem}\label{veblenhierarchy} Assume $otype(T) = \alpha \in \varepsilon_0 \backslash \{0\}$. Then $(\varphi_T0,<)$ is a notation system for the ordinal $\omega^{\omega^{-1 + \alpha}}$.
\end{theorem}
\begin{proof}
A proof of this fact can be found in \cite{binarytreesgyesik}.
\end{proof}

\subsubsection{Using the $\pi_i$-collapsing functions}

We use an ordinal notation system that employs the $\pi_i$-collapsing functions. These functions are based on Buchholz's $\Psi_i$-functions \cite{BuchholzAPAL1986}. We state some basic facts that the reader can find in \cite{BuchholzAPAL1986,schuttesimpson}.

\begin{definition}\label{def:Omega_i}
Let $\Omega_0:=1$ and define $\Omega_i$ as the $i^{\text{th}}$ regular ordinal number strictly above $\omega$. Define $\Omega_\omega$ as $\sup_i \Omega_i$.
\end{definition}

Define the sets $B_i^m(\alpha)$ and $B_i(\alpha)$ and the ordinal numbers $\pi_i\alpha$ as follows.

\begin{definition}
\begin{itemize}[itemsep=0ex]
\item If $\gamma = 0$ or $\gamma < \Omega_i$, then $\gamma \in B^m_i(\alpha)$,
\item if $i\leq j$, $\be<\al$, $\be \in B_j(\beta)$ and $\be \in B_i^m(\al)$, then $\pi_j\be \in B_i^{m+1}(\al)$,
\item define $B_i(\al)$ as $\bigcup_{m<\omega} B_i^m(\al)$,
\item $\pi_i\al := \min\{\eta : \eta \notin B_i(\al)\}$.
\end{itemize}
\end{definition}

\begin{lemma}
\begin{enumerate}[itemsep=0ex]
\item if $i\leq j$ and $\al \leq \be$, then $B_i(\al) \subseteq B_j(\be)$ and $\pi_i\al \leq \pi_j\be$,
\item $\Omega_i \leq \pi_i\al < \Omega_{i+1}$,
\item $\pi_i0 = \Omega_i$,
\item $\al \in B_i(\al) $ and $\al < \be$ yields $\pi_i \al < \pi_i\be$,
\item $\al \in B_i(\al) $, $\be \in B_i(\be)$ and $\pi_i \al = \pi_i \be$ yields $\al = \be$.
\end{enumerate}
\end{lemma}

\begin{definition}
For ordinals $\alpha \in B_0(\Omega_\omega)$, define $G_i(\pi_j\alpha)$ as
\begin{equation*}
\begin{cases}
\emptyset  & \text{ if $j<i$,}\\
G_i\al\cup\{\al\} &\text{ otherwise.}
\end{cases}
\end{equation*}
Define $G_i(0)$ as $\emptyset$.
\end{definition}

This is well-defined, because one can prove that $\pi_j \alpha \in B_0(\Omega_\omega)$ yields $\alpha \in B_0(\Omega_\omega)$.

\begin{notation}
For a set of ordinals $A$ and an ordinal $\alpha$, we write $A < \alpha$ if for all $\beta \in A (\beta < \alpha)$.
\end{notation}

\begin{lemma}
If $\alpha \in B_0(\Omega_\omega)$, then $G_i(\alpha) < \beta$ iff $\alpha \in B_i(\beta)$.
\end{lemma}
\begin{proof} We prove this by induction on the length of construction of $\alpha$. If $\alpha=0$ or $\alpha = \pi_j\delta$ with $j<i$, then this is trivial. Assume $\alpha = \pi_j\delta$ with $j\geq i$. $\alpha = \pi_j \delta \in B_0(\Omega_\omega)$ yields $\delta \in B_j(\delta)$. Now, $G_i(\alpha) < \beta$ is valid iff $G_i(\delta) < \beta$ and $\delta < \beta$. By the induction hypothesis, this is equivalent with $\delta \in B_i(\beta)$ and $\delta < \beta$, which is equivalent with $\alpha = \pi_j \delta \in B_i(\beta)$ because $\delta \in B_j(\delta)$.
\end{proof}

\medskip

Now we define the ordinal notation systems $\pi(\omega)$ and $\pi(n)$, but first, we have to define a set of terms $\pi(\omega)'$ and $\pi(n)'$.

\begin{definition}
\begin{itemize}
\item $0 \in \pi(\omega)'$ and $0 \in \pi(n)'$,
\item if $\alpha \in \pi(\omega)'$, then $D_j \alpha \in \pi(\omega)'$,
\item if $\alpha \in \pi(n)'$ and $j<n$, then $D_j \alpha \in \pi(n)'$.
\end{itemize}
\end{definition}

\begin{definition} Let $\alpha, \beta \in \pi(\omega)'$ or $\alpha, \beta \in \pi(n)'$. Then define $\alpha < \beta$ if 
\begin{enumerate}
\item $\alpha = 0$ and $\beta \neq 0$,
\item $\alpha = D_j \alpha'$, $\beta = D_k \beta'$ and $i<j$ or $i=j$ and $\alpha' < \beta'$.
\end{enumerate}
\end{definition}

\begin{lemma}
$<$ is a linear order on $\pi(\omega)'$ and $\pi(n)'$.
\end{lemma}
\begin{proof}
Similar as Lemma 2.1 in \cite{BuchholzAPAL1986}.
\end{proof}

\begin{definition}
For $\alpha \in \pi(\omega)', \pi(n)'$, define $G_i(\alpha)$ as follows.
\begin{enumerate}
\item $G_i(0) = \emptyset$,
\item $G_i(D_j \alpha') :=
\begin{cases}
 G_i(\alpha') \cup \{ \alpha' \} & \mbox{if } i \leq j,\\
 \emptyset & \mbox{if } i>j.
\end{cases}
$
\end{enumerate}
\end{definition}

Now, we are ready to define to ordinal notation systems $\pi(\omega) \subseteq \pi(\omega)'$ and $\pi(n) \subseteq \pi(n)'$.

\begin{definition} $\pi(\omega)$ and $\pi(n)$ are the least sets such that
\begin{enumerate}
\item $0 \in \pi(\omega)$, $0 \in \pi(n)$,
\item if $\alpha \in \pi(\omega)$ and $G_i(\alpha) < \alpha$, then $D_i \alpha \in \pi(\omega)$,
\item if $\alpha \in \pi(n)$, $i<n$ and $G_i(\alpha) < \alpha$, then $D_i \alpha \in \pi(n)$.
\end{enumerate}

\end{definition}

Apparently, the $D_j \alpha$'s correspond to the ordinal functions $\pi_j \alpha$:

\begin{definition} For $\alpha \in \pi(\omega)$ and $\pi(n)$, define
\begin{enumerate}
\item $o(0) := 0$,
\item $o(D_j \alpha') := \pi_j(o(\alpha'))$.
\end{enumerate}
\end{definition}

\begin{lemma} For $\alpha, \beta \in \pi(\omega)$ or $\pi(n)$, we have:
\begin{enumerate}
\item $o(\alpha) \in B_0(\Omega_\omega)$,
\item $G_i(o(\alpha)) = \{ o(x) : x \in G_i(\alpha)\}$,
\item $\alpha < \beta \rightarrow o(\alpha) < o(\beta)$.
\end{enumerate}
\end{lemma}
\begin{proof}
A similar proof can be found in \cite{BuchholzAPAL1986}.
\end{proof}

\begin{lemma}
\begin{enumerate}[itemsep=0ex]
\item $\{o(x) : x \in \pi(\omega )\} = B_0(\Omega_\omega)$,
\item $\{o(x) : x \in \pi(\omega)$ and $x < D_10\} = \pi_0 \Omega_\omega$,
\item $\{o(x) : x \in \pi(n)$ and $x < D_10\} =  \pi_0\Omega_n$ if $n>0$.
\end{enumerate}
\end{lemma}
\begin{proof}
A similar proof can be found in \cite{BuchholzAPAL1986}.
\end{proof}

\medskip

Define $\pi(\omega) \cap D_10$ as $\pi_0(\omega)$ and $\pi(n) \cap D_10$ as $\pi_0(n)$. 
It is very important to see that we work with two different contexts: one context is at the level of ordinals, i.e. if we use the $\pi_i$'s. The other context at the syntactical level, i.e. if we use the $D_i$'s (because it is an ordinal notation system). 
The previous results actually indicate that $D_i$ and $\pi_i$ play the same role and for notational convenience, we will identify these two notations: from now on, we write $\pi_i$ instead of $D_i$. The context will make clear what we mean. 
If we use $\Omega_i$ in the ordinal context, it is interpreted as in Definition \ref{def:Omega_i}. In the other context, at the level of ordinal notation systems, we define $\Omega_i$ as $D_i0$ (which is now also denoted by $\pi_i0$). 

\medskip

We could also have defined $\pi(\omega)$ in the following equivalent way.

\begin{definition} Define $\pi(\omega)$ as the least set of ordinals such that
\begin{enumerate}[itemsep=0ex]
\item $0 \in \pi(\omega)$,
\item If $\al \in \pi(\omega)$ and $\al \in B_i(\al)$, then $\pi_i\al \in \pi(\omega)$.
\end{enumerate}
Define $\pi(n)$ in the same manner, but with the restriction that $i<n$.
\end{definition}

In \cite{schuttesimpson}, the following theorem is shown. Therefore, $\pi_0(n)$ is an ordinal notation system for $\omega_n[1]$ if $n>0$ and $\pi_0(\omega)$ is a system for $\varepsilon_0$.

\begin{theorem}\label{ordertypepi}
\begin{enumerate}[itemsep=0ex]
\item $\pi_0 \Omega_n = \omega_n[1]$ if $n>0$,
\item $\pi_0 \Omega_\omega = \varepsilon_0$.
\end{enumerate}
\end{theorem}

\subsubsection{Using the $\vartheta_i$-collapsing functions}

In this subsection, we give an ordinal representation system that is based on the $\vartheta_i$-functions. For more information about this system that includes the addition-operator, see \cite{PohlersinLeeds,rathjenweiermann}. In this subsection, we introduce them without the addition-operator.


\begin{definition} Define $T$ and the function $S$ simultaneously as follows. $T$ is the least set such that $0 \in T$, where $S(0):=-1$ and if $\alpha \in T$ with $S(\alpha)\leq i+1$, then $\vartheta_i\alpha\in T$ and $S(\vartheta_i \alpha):= i$. We call the number of occurrences of symbols $\vartheta_j$ in $\alpha \in T$, the \textbf{length} of $\alpha$ and denote this by $lh(\al)$. Furthermore, let $\Omega_i:= \vt_i 0$.
\end{definition}

Like in the $D_i$-case, $\Omega_i$ is defined as something syntactically because $T$ is an ordinal notation system. However, the usual interpretation of $\Omega_i$ in the context of ordinals is as in Definition \ref{def:Omega_i}. $S(\alpha)$ represents the index $i$ of the first occurring $\vartheta_i$ in $\alpha$, if $\alpha \neq 0$.

\begin{definition} Let $n<\omega$. Define $T_n$ as the set of elements $\alpha$ in $T$ such that for all $\vartheta_j$ in $\alpha$, we have $j<n$. Let $T[m]$ be the set of elements $\alpha$ in $T$ such that $S(\alpha) \leq m$. Define $T_n[m]$ accordingly.
%
\end{definition}

For example $T_1 = T_1[0] = \{0,\vartheta_00, \vartheta_0\vartheta_00,\dots\}$.  For every element $\alpha$ in $T$, we define its \textit{coefficients}. The definition is based on the usual definition of the coefficients in a notation system \textit{with} addition.

\begin{definition} Let $\alpha \in T$. If $\alpha=0$, then $k_i(0):=0$. Assume $\alpha = \vartheta_j(\beta)$. Let $k_i(\alpha)$ then be
\[\begin{cases}
\vartheta_j(\beta) = \al & \text{if $j\leq i$,}\\
k_i(\beta) & \text{if $j>i$.}
\end{cases}\]
\end{definition}

Using this definition, we introduce a well-order on $T$ (and its substructures).
This ordering is based on the usual ordering between the $\vartheta_i$-functions defined with addition.

\begin{definition}
\begin{enumerate}[itemsep=0ex]
\item If $\al\not=0$, then $0<\al$,
\item if $i<j$, then $\vt_i\al<\vt_j\be$,
\item if $\al<\be$ and $k_i\al<\vt_i\be$, then $\vt_i\al<\vt_i\be$,
\item if $\al>\be$ and $\vt_i\al\leq k_i\be$, then $\vt_i\al<\vt_i\be$.
\end{enumerate}
\end{definition}

\begin{definition}
If $\alpha, \beta \in T$ and $\beta < \Omega_1$, let $\alpha[\beta]$ be the element in $T$ where the last zero in $\alpha$ is replaced by $\beta$. 
\end{definition}

The following lemma gives some useful properties of this ordinal notation system.

\begin{lemma}\label{propertiesthetafunction} For all $\alpha,\be$ and $\ga$ in $T$ and for all $i<\omega$,
\begin{enumerate}[itemsep=0ex]
\item $ k_i(\al) \leq \al$,
\item if $\alpha = \vartheta_{j_1} \dots \vartheta_{j_n} t $ with $j_1,\dots, j_n \geq i$ and ($t = 0$ or $t =\vt_k t'$ with $k\leq i$), then $t < \vartheta_i(\alpha)$,
\item $k_i(\al) < \vt_i\al$,
\item $k_i(\al)[\gamma] = k_i(\al[\gamma])$ for $\gamma < \Omega_1$,
\item if $\ga < \Omega_1$, then $\ga \leq \be[\ga]$ and there is only equality if $\be =0$,
\item if $\al<\be$ and $\ga<\Om_1$, then $\al[\ga]<\be[\ga]$.
\end{enumerate}
\end{lemma}
\begin{proof}
\begin{enumerate}
\item The first assertion is easy to see.
\item By induction on $lh(\al)$ and sub-induction on $lh(t)$. If $\al = 0$, then the claim is trivial. Assume from now on $\al >0$. If $t=0$ or $t = \vt_k t'$ with $k<i$, then this is trivial. Assume $t= \vt_i t'$. Then $t= \vt_i \vt_{l_1} \ldots \vt_{l_m} k_i(t')$ with $l_1,\dots, l_m>i$. The sub-induction hypothesis, $lh(k_i(t')) < lh(t)$ and $\alpha =  \vartheta_{j_1} \dots \vartheta_{j_n} \vt_i \vt_{l_1} \ldots \vt_{l_m} k_i(t')$ yield $k_i(t') < \vt_i \al$. If $t' < \alpha$, then $t = \vt_i t'< \vt_i \al$. Assume $t' > \alpha$. Note that equality is impossible because $t'$ is a strict subterm of $\al$. 
We claim that $t = \vt_i t' \leq k_i(\al)$, hence we are done. We know that $k_i(\al) = \vartheta_{j_p} \ldots \vt_{j_n} \vt_i t'$ for a certain $p$ with $j_p=i$ or $k_i(\alpha) = \vt_i t'$. In the latter case, the claim is trivial. In the former case, the main induction hypothesis on $ \vartheta_{j_{p+1}} \ldots \vt_{j_n} \vt_i t'  $ yields $t < \vt_i \vartheta_{j_{p+1}} \ldots \vt_{j_n} \vt_i t' = k_i(\al) $.
\item This follows easily from the second assertion because $\alpha = \vt_{j_1} \ldots \vt_{j_n} k_i(\al)$ with $j_1,\dots,j_n>i$.
\item Follows easily by induction on $lh(\al)$.
\item By induction on $lh(\ga)$ and sub-induction on $lh(\be)$. If $\ga=0$, the statement is trivial to see. From now on, let $\ga = \vt_0 \ga'$. If $\be = 0$ or $\be = \vt_i \be'$ with $i>0$, the statement also easily follows. Assume $\be = \vt_0 \be'$. We see $\be[\ga] = \vt_0(\be'[\ga])$.
Suppose $\ga' < \be'[\ga]$. Assume $\ga' = \vt_{j_1} \dots \vt_{j_k} k_0(\ga')$ with $j_1,\dots,j_k >0$ and define $\overline{\be}$ as $\be[\vt_0\vt_{j_1} \dots \vt_{j_k}0]$. The main induction hypothesis yields $k_0(\ga') \leq \overline{\be}[k_0(\ga')] = \be[\ga] = \vt_0(\be'[\ga])$. Note that equality is not possible because $k_0(\ga')$ is a strict subterm of $\overline{\be}[k_0(\ga')]$, hence $\ga = \vt_0\ga' < \vt_0(\be'[\ga]) = \be[\ga]$.
Assume $\ga' > \be'[\ga]$. The sub-induction hypothesis yields $\ga \leq k_0(\be')[\ga]  \stackrel{\gamma < \Omega_1}{=}k_0(\be'[\ga])$.
Hence, $\ga \leq k_0(\be'[\ga]) < \vt_0(\be'[\ga]) = \be[\ga]$.
\item By induction on $lh(\al)+ lh(\be)$. If $\alpha = 0$ and $\beta\neq 0$, then the previous assertion yields $\al[\ga]= \ga < \be[\ga]$.
 Assume $\al = \vt_i \al' < \vt_j\be' = \be$. If $i<j$, then also $\al[\ga]<\be[\ga]$. Suppose $i=j$. Then either $\al'<\be'$ and $k_i(\al') < \vt_j\be'$, or $\al \leq k_j(\be')$. In the former case, the induction hypothesis yields $\al'[\ga]<\be'[\ga]$ and $k_i(\al'[\ga]) \stackrel{\ga < \Omega_1}{=} k_i(\al')[\ga] <  (\vt_j\be')[\ga] = \vt_j(\be'[\ga])$. Hence, $\al[\ga] = (\vt_i \al')[\ga] = \vt_i(\al'[\ga]) < \vt_j(\be'[\ga]) = (\vt_j\be')[\ga] = \be[\ga]$.
In the latter case, the induction hypothesis yields $\al[\ga] \leq k_j(\be')[\ga] \stackrel{\ga < \Omega_1}{=} k_j(\be'[\ga]) < \vartheta_j(\be'[\ga]) = (\vt_j\be')[\ga] = \be[\ga]$.
\end{enumerate}
\end{proof}

\medskip

On $T$ and its substructures, we define the following partial order $\unlhd$, which can be seen as a natural sub-order of the ordering $<$ on $T$ (see Lemma \ref{lemma:unlhdsubordering<}). 

\begin{definition}
\begin{enumerate}[itemsep=0ex]
\item $0\unlhd \al$,
\item if $\al\unlhd k_i\be$, then $\al \unlhd \vt_i\be$,
\item if $\al\unlhd\be$, then $\vt_i\al\unlhd \vt_i\be$.
\end{enumerate}
\end{definition}

Apparently, $T_n$ with this natural sub-ordering is the same as $\overline{\S}_n^{s}$.

\begin{lemma}\label{lemma:Tnlhdisthegapconidition}
$(T_n,\unlhd) \cong (\overline{\S}_n, \leq_{gap}^s)$.
\end{lemma}
\begin{proof}
Define $e:T_n \to \overline{\S}_n$ as follows. $e(0)$ is the empty sequence $\varepsilon$. Let $e(\vartheta_i\alpha)$ be $(i)^\frown e(\alpha)$. For example $e(\vartheta_2\vartheta_10)$ is the finite sequence $21$. It is trivial to see that $e$ is a bijection. So the only thing we still need to show is that for all $\al$ and $\be$ in $T_n$, $e(\al)\leq_{gap}^s e(\be)$ if and only if $\al \unlhd \be$. We show this by induction on the sum of the lengths of $\alpha$ and $\beta$.
If $\alpha$ or $\beta$ are equal to $0$, then this is trivial. Assume $\alpha$ and $\beta$ are different from $0$. Hence, $\alpha = \vartheta_i \alpha'$ and $\beta = \vartheta_j \beta'$.
Assume $\alpha \unlhd \beta$. Then $\alpha \unlhd k_j \beta'$ or $i=j$ and $\alpha' \unlhd \beta'$. In the latter case, the induction hypothesis yields $e(\alpha') \leq_{gap}^s e(\beta')$, hence $e(\alpha) = (i)^\frown e(\alpha') \leq_{gap}^s (i)^\frown e(\beta') = e(\beta)$. In the former case, assume $\beta' = \vartheta_{l_1} \dots \vartheta_{l_k} \beta''$, with $l_1,\dots,l_k >j$ and $S(\beta'')\leq j$ such that $k_j(\beta')=\beta''$. The induction hypothesis yields $e(\alpha) \leq_{gap}^s e(\beta'')$. From the strong gap-embeddability relation we obtain $i \leq S(\beta'') \leq j$, hence $e(\alpha) \leq_{gap}^s (jl_1,\dots l_k)^\frown e(\beta'')$ because $j,l_1,\dots,l_k \geq i$. The reverse direction can be proved in a similar way.
\end{proof}

\bigskip

The previous proof also yields $(T_n[0],\unlhd) \cong (\overline{\S}_n[0], \leq_{gap}^s) = (\overline{\S}_n[0], \leq_{gap}^w)$. We prove that the linear order $<$ on $T_n$ is a linear extension of $\lhd$. Let $\alpha \lhd \beta$ if $\al \unlhd \be$ and $\al \neq \be$.

\begin{lemma}\label{lemma:unlhdsubordering<}
If $\alpha \unlhd \be$, then $\al \leq \be$.
\end{lemma}
\begin{proof}
We prove this by induction on the sum of the lengths of $\al$ and $\be$
Assume $\al \unlhd \be$. If $\al = 0$, then trivially $\al \leq \be$. Assume $\al = \vt_i \al'$. $\al \unlhd \be$ yields $\be = \vt_i \be'$ and either $\al \unlhd k_i\beta'$ or $\al' \unlhd \be'$. In the first case, the induction hypothesis yields $\al \leq k_i \beta' < \vt_i \beta' = \beta$. Assume that $\al' \unlhd \be'$. The induction hypothesis yields $\al' \leq \be'$. if $\al' = \be'$, we can finish the proof, so assume $\al' < \be'$. We want to prove that $k_i\al' < \be$. Using the induction hypothesis, it is sufficient to prove that $k_i\al' \lhd \be$. This follows from $\al = \vt_i \vt_{j_1} \dots \vt_{j_l} k_i\al' \unlhd \be$ (with $j_1,\dots,j_l>i$) and Lemma \ref{lemma:Tnlhdisthegapconidition}.
\end{proof}

\bigskip

The previous lemmas imply that the linear ordering on $T_n[0]$ is a linear extension of $\overline{\S}_n[0]$ with the strong (and weak) gap-embeddability relation and furthermore,
\[o(T_n[0],\unlhd) = o(\overline{\S}^{s}_n[0]) = o(\overline{\S}^{w}_n[0]) = o(\overline{\S}^{w}_n).\]


These results also hold in the case if we allow the addition-operator:
the ordinal notation systems using $\vartheta_i$ and the addition-operator corresponds to a linear extension of Friedman's $\wpo$ $\overline{\mathbb{T}}_n[0]$ with the strong and weak gap-embeddability relation ($\overline{\mathbb{T}}_n[0]$ is defined in a similar way as $\overline{\S}_n[0]$, but with trees). It is our general belief that this is a \textbf{maximal} linear extension. In \cite{anordertheoreticcharacterizationofthehowardbachmannhierarchy, wellpartialorderingsandthebigveblennumber} we already obtained partial results concerning this conjecture. In this paper, we want to investigate whether this is also true for the linearized version of the gap-embeddability relation, i.e. if the well-order $(T_n[0],<)$ is a maximal linear extension of $(T_n[0],\unlhd)  \cong (\overline{\S}_n[0],\leq^s_{gap}) = (\overline{\S}_n[0],\leq^w_{gap})$. This can be shown by proving that the order type of $(T_n[0],<)$ is equal to the maximal order type of $ (\overline{\S}_n[0],\leq^s_{gap})$, which is $\omega_{2n-1}$. 

\medskip

Quite surprisingly, the maximal linear extension principle is not true in this sequential version: if $n>2$, then the order type of $(T_n[0],<)$ is equal to $\omega_{n+1}$. We remark that the maximal linear extension principle is true if $n=1$ and $n=2$. We prove these claims in the next sections.


\section{Maximal linear extension of gap-sequences with one and two labels}

It is trivial to show that the order type of $(T_1[0],<)$ is equal to $\om$, hence $(T_1[0],<)$ corresponds to a maximal linear extension of $\overline{\S}^s_1[0]$. So we can concentrate on the case of $T_2[0]$. We show that the order type of $(T_2[0],<)$ is equal to $\omega^{\omega^\omega}$. This implies that ($T_2[0],<)$ corresponds to a maximal linear extension of $\overline{\S}^w_2[0]$ and that the order type of $(T_2[0],<)$ is equal to $o(\overline{\S}^w_2)$. More specifically, we show that
\[ \sup_{n_1,\dots,n_k} \vartheta_0 \vartheta_1^{n_1}\dots \theta_0\vartheta_1^{n_k} (0) = \omega^{\omega^\omega}.\]
The supremum is equal to $\vartheta_0\vartheta_1\vartheta_2(0)$ and knowing that $\Omega_i$ is defined as $\vartheta_i(0)$, we thus want to show
\[ \vartheta_0 \vartheta_1 \Omega_2 = \omega^{\omega^\omega}.\]

\begin{theorem}\label{lineartheta2labelsproof1} $\vt_0\vt_1\Omega_2 =\om^{\om^\om}$
\end{theorem}
\begin{proof} We present a order-preserving bijection from $\varphi_\omega 0$ to $\vt_0 \vt_1 \Omega_2$. Lemma \ref{veblenhierarchy} then yields the assertion.

Define $\chi0:=0$ and $\chi\varphi_n\al:=\vt_0\vt_1^n\chi\al$. Then $\chi$ is order preserving.
Indeed, we show $\al<\be\imp\chi\al<\chi\be$ by induction on $lh(\al)+lh(\be)$.
If $\al=0$ and $\be\not=0$, then trivially $\chi\al<\chi\be$.
Let $\al=\varphi_n\al'<\be=\varphi_m\be'$.
If $\al'<\be$ and $n<m$ then the induction hypothesis yields $\chi\al'<\vt_0\vt_1^m\chi\be'$ and then
$n<m$ yields  $\chi\al=\vt_0\vt_1^n\chi\al'<\vt_0\vt_1^m\chi\be'=\chi\be.$
If $n=m$ and $\al'<\be'$ then $\chi\al=\vt_0\vt_1^n\chi\al'<\vt_0\vt_1^n\chi\be'=\chi\be$.
If $\al\leq\be'$, then $\chi\al\leq\chi\be' <\vt_0\vt_1^m\chi\be'.$
\end{proof}

\medskip

It might be instructive, although it is in fact superfluous, to redo the argument for the standard representation for $\om^{\om^\om}$. First, we need an additional lemma.

\begin{lemma}\label{lemmaforlineartheta2labels}
Let $\al,\be$ and $\ga$ be elements in $T$.
\begin{enumerate}[itemsep = 0ex]
\item $\al<\be<\Om_1$ and $l_i<n$, $k_i>0$ for all $i\leq r$ yield\\
$\vt_0^{k_0}\vt_1^{l_1}\vt_0^{k_1}\ldots  \vt_1^{l_r}\vt_0^{k_r}\vt_1^n\al    < \vt_0\vt_1^n\be$,
\item $\al<\be<\Om_1$ and $l_{ij}<n$, $k_{ij}>0$ for all $i,j$ yield\\
$\vt_0^{k_{00}}\vt_1^{l_{01}}\vt_0^{k_{01}}\ldots \vt_1^{l_{0m_0}}\vt_0^{k_{0m_0}}\vt_1^n\ldots  \vt_0^{k_{r0}}\vt_1^{l_{r1}}  \vt_0^{k_{r1}}\ldots \vt_1^{l_{rm_r}} \vt_0^{k_{rm_r}}\vt_1^n\al   < $\\
$\vt_0^{p_{00}}\vt_1^{q_{01}}\vt_0^{p_{01}} \ldots \vt_1^{q_{0s_0}}\vt_0^{p_{0s_0}}\vt_1^n\ldots  \vt_0^{p_{r0}}\vt_1^{q_{r1}}\vt_0^{p_{r1}}\ldots \vt_1^{q_{rs_r}}\vt_0^{p_{rs_r}}\vt_1^n\be$,
\item  $l_i<n$ and $k_i>0$ for all $i\leq r$ yield $\vt_0^{k_0}\vt_1^{l_1}\vt_0^{k_1}\ldots  \vt_1^{l_r}\vt_0^{k_r}0   < \vt_0\vt_1^n0$.
\end{enumerate}
\end{lemma}
\begin{proof}
The first assertion follows by induction on $r$:
if $r=0$, then $\vt_0^{k_0}\vt_1^{n}\al<\vt_0\vt_1^n\be$ follows by induction on $k_0$.
If $r>0$, then the induction hypothesis yields $\xi=\vt_0^{k_1}\ldots  \vt_1^{l_r}\vt_0^{k_r}\vt_1^n\al    < \vt_0\vt_1^n\be$.
We have $\xi<\vt_1^{n-l_1}\be$ because $k_1>0$, and thus $\vt_1^{l_1}\xi<\vt_1^n\be$. We prove $\vt_0^{k_0}\vt_1^{l_1}\xi<\vt_0\vt_1^n\be$ by induction on $k_0$. First note that we know $k_0(\vte^{l_1}\xi) = \xi < \vtn\vte^n\be$, hence the induction base $k_0=1$ easily follows. The induction step is straightforward.\\
The second statement follows from the first by induction on the number of involved blocks.\\
The third assertion follows by induction on $r$.
\end{proof}

\begin{theorem}\label{lineartheta2labelsproof2}
$\om^{\om^\om}=\vt_0\vt_1\Omega_2$
\end{theorem}
\begin{proof}
Define $\chi:\om^{\om^\om}\to \vt_0\vt_1\Omega_2$ as follows.
Take $\alpha < \om^{\om^\om}$. Let $n$ be the least number such that $\alpha < \om^{\om^n}$. Let $m$ then be minimal such that
\[\al=\om^{\om^{n-1}\cdot m}\cdot \al_m+\cdots+ \om^{\om^{n-1}\cdot 0}\cdot \al_0,\]
with $\al_m\not=0$ and $\al_0,\dots,\al_m <\om^{\om^{n-1}}$. Put $\chi \al$ as the element
\[\vt_0\vt_1^n\chi(\al_0)\cdots \vt_0\vt_1^n\chi(\al_m).\]
It is trivial to see that $\chi$ is surjective.
We claim that $\al<\be$ yields $\chi(\al)<\chi(\be)$. We prove the claim by induction on $lh(\al)+lh(\be)$.\\
Let $\al=\om^{\om^{n-1}\cdot m}\cdot \al'+\tilde{\al}$ and $\be=\om^{\om^{n'-1}\cdot m'}\cdot \be'+\tilde{\be}$ with $\al',\be'>0$, $\tilde{\al}< \om^{\om^{n-1}\cdot m}$ and $\tilde{\be}< \om^{\om^{n'-1}\cdot m'}$.
If $n<n'$, then $\chi(\be)$ contains a consecutive sequence of $\vt_1^{n'}$ which has no counterpart in $\chi(\al)$. Hence, $\chi\al<\chi\be$ follows from a combination of the second and third assertion of the previous lemma.
If $n=n'$ and $m<m'$ then $\chi(\be)$ contains at least one more consecutive sequence of $\vt_1^{n}$ than the ones occurring in $\chi(\al)$. Thus again $\chi\al<\chi\be$ using the second and third assertion of the previous lemma.
If $n=n'$ and $m=m'$ and $\al'<\be'$ then the induction hypothesis yields $\chi(\al')<\chi(\be')$. We know $\chi(\al)=\chi(\alt)\vtn\vte^n\chi(\als)$ and $\chi(\be)=\chi(\bet)\vtn\vte^n\chi(\bes)$. So, the second assertion of the previous lemma yields the assertion.
If $n=n'$ and $m=m'$ and $\al'=\be'$ then $\tilde{\al}<\tilde{\be}$ and the induction hypothesis yield $\chi(\tilde{\al})<\chi(\tilde{\be})$ and $\chi(\al)=\chi(\alt)\vtn\vte^n\chi(\als)$ and $\chi(\be)=\chi(\bet)\vtn\vte^n\chi(\bes)$. The assertion follows from the sixth assertion of Lemma \ref{propertiesthetafunction}.
\end{proof}


\section{The order type of $(T_n[0],<)$ with $n>2$}

As mentioned before, we expected that $(T_n[0],<)$ corresponds to a maximal linear extension of $\overline{\S}_n^w[0]$ and $\overline{\S}_n^s[0]$. This could have been shown by proving that the order type of $(T_n[0],<)$ is equal to $\omega_{2n-1}$. 
However, by calculations of the second author, we saw that $(T_n[0],<)$ does not correspond to a maximal linear extension. Instead we now show that the order type of $(T_{n}[0],<)$ is equal to $\om_{n+1}$ for $n\geq 2$. 
 It is straightforward to prove that the order type of $(T_{n+1}[0],<)$ is equal to $\vt_0\vt_1\vt_2\ldots\vt_n\Omega_{n+1}$, hence we will show that 
\[\omega_{n+2} = \vt_0\vt_1\vt_2\ldots\vt_n\Omega_{n+1},\]
for $n\geq 1$. 
To prove the lower bound ($\leq$), we use results by Sch\"utte and Simpson \cite{schuttesimpson}. The other direction will be shown by turning the already convincing sketch of the second author into a general argument.


\subsection{Lower bound}
In this subsection, we prove $\om_{n+2} \leq\vt_0\vt_1\vt_2\ldots\vt_n\Omega_{n+1}$, where $n \geq 1$.

\begin{definition}
\begin{enumerate}
\item If $\alpha \in  T$, define
\begin{equation*}
\di\al:=\begin{cases}
\vti\al  & \text{ if $S\al\leq i$,}\\
 \vartheta_i \dipe \al &\text{ otherwise.}
\end{cases}
\end{equation*}
\item For ordinals in $\pi(\omega)$, define $\overline{\cdot}$ as follows:
\begin{itemize}
\item $\ol{0}:=0$,
 \item $\ol{\pi_i\al}:=d_{i+1}\oa$.
 \end{itemize}
 \item On $T$, define $0[\be]:=\be$ and $(\vti\al)[\be]:=\vti(\al[\be])$.
 \item Let $\psi$ be the function from $\varphi_{\pi_0(n)}0$ to $T$ which is defined as follows:
 \begin{itemize}
 \item $\psi0:=0$,
 \item $\psi\varphi_{\pi_0\al}\be:=d_0\oa[\psi\be]$.
 \end{itemize}
\end{enumerate}
\end{definition}

It is easy to see that the image of $\psi$ lies in $T_{n+1}[0]$. We show that $\psi$ is order-preserving in order to obtain a lower bound for the order type of $T_{n+1}[0]$.

\begin{lemma}\label{lemma:psiisorderpreserving} Let $\alpha, \beta$ be elements in $\pi(\omega)$ and $\gamma, \delta$ elements in $T$.
\begin{enumerate}
\item
$\al<\be$ and $\ga,\de<\Om$ yields $\oa[\ga]<\ob[\de]$,
\item  $\ga<\de<\Om$ yields $\oa[\ga]<\oa[\de]$,
\item $G_k\al<\be$ and $\ga,\de<\Om$ yield $ k_{k+1}\oa[\ga]<d_{k+1}\ob[\de],$
\item $\alpha < \beta$, $G_k\al<\be$ and $\ga,\de<\Om$ yields $ d_{k+1}\oa[\ga]<d_{k+1}\ob[\de],$
\item If $\zeta, \eta \in \varphi_{\pi_0(n)}0$,  then $\zeta<\eta $ yields $\psi\zeta<\psi\eta$.
\end{enumerate}
\end{lemma}
\begin{proof}
We prove assertions $1.$--$4.$ simultaneously by induction on $lh(\alpha)$. If $\alpha = 0$, then $1.$ and $2.$ are trivial. Assertion $3.$ is also easy to see because $k_{k+1} \overline{\alpha} [\gamma] = \gamma < \Omega \leq d_{k+1} \overline{\beta}[\delta]$. In assertion $4.$, $d_{k+1} \overline{\alpha} [\gamma] = \vartheta_{k+1} \gamma$. Now,  $d_{k+1} \overline{\beta}[\delta] = \vartheta_{k+1} \zeta$ for a certain $\zeta \geq \Omega$. Therefore, $\gamma < \zeta$ and $k_{k+1}\gamma = \gamma < d_{k+1} \overline{\beta}[\delta]$, which yields $d_{k+1} \overline{\alpha} [\gamma] = \vartheta_{k+1} \gamma < d_{k+1} \overline{\beta}[\delta]$.

\medskip

From now on, assume $\alpha  = \pi_i \alpha'$.

\medskip

Assertion $1.$: $\alpha < \beta$ yields $\beta = \pi_j \beta'$ with $i \leq j$. If $i<j$, then the assertion follows. Assume $i = j$. Then $\alpha' < \beta'$. We know that $G_i(\alpha') < \alpha'$ because $\pi_i \alpha' \in \pi(\omega)$. Assertion 4. and $\alpha' < \beta'$ yield $d_{i+1} \overline{\alpha'}[\gamma] < d_{i+1} \overline{\beta'} [\delta]$, which is $\overline{\alpha}[\gamma] < \overline{\beta} [\delta]$.


\medskip

Assertion $2.$: We know that $G_i(\alpha') < \alpha'$, hence $G_l(\alpha' ) < \alpha'$ for all $l \geq i$. Assertion $3.$ then yields $k_{l+1} \overline{\alpha'}[\gamma] < d_{l+1}  \overline{\alpha'} [\delta]$ for all $l \geq i$.
If $\alpha' = 0$, then assertion $2.$ easily follows from $\gamma < \delta$. Assume $\alpha' \neq 0$.

If $S(\overline{\alpha'}) \leq i+1$, then $\overline{\alpha}[\gamma] = d_{i+1} \overline{\alpha'} [\gamma] = \vartheta_{i+1} \overline{\alpha'} [\gamma]$. Therefore, assertion $2.$ follows if $\overline{\alpha'}[\gamma] < \overline{\alpha'}[\delta]$ and $k_{i+1}\overline{\alpha'}[\gamma] <  \vartheta_{i+1} \overline{\alpha'} [\delta] = d_{i+1} \overline{\alpha'} [\delta]$. We already know that the second inequality is valid. The first inequality follows from the main induction hypothesis.

Assume now $S(\overline{\alpha'}) > i+1$. We claim that $d_j\overline{\alpha'}[\gamma] < d_j \overline{\alpha'}[\delta]$ for all $j \in \{ i+1, \dots, S(\overline{\alpha'}) \}$. Assertion $2.$ then follows from $j=i+1$.  We prove our claim by induction on $l= S(\overline{\alpha'}) - j \in \{0, \dots , S(\overline{\alpha'})-i-1\}$.
If $l = 0$, then $j= S(\overline{\alpha'})  > i+1$. Then the claim follows if $k_j \overline{\alpha'}[\gamma] < d_j \overline{\alpha'}[\delta]$ and $\overline{\alpha'}[\gamma] < \overline{\alpha'}[\delta]$. The first inequality follows from assertion $3.$ and the fact that $G_{j-1}(\alpha') < \alpha'$. The second inequality follows from the main induction hypothesis.
Now, assume that the claim is true for $l$. We want to prove that it is true for $l+1 = S(\overline{\alpha'}) -j$. Hence, $l = S(\overline{\alpha'}) - (j+1)$. The induction hypothesis yields $d_{j+1}\overline{\alpha'}[\gamma] < d_{j+1} \overline{\alpha'}[\delta]$. We also see that $j \geq i+1$, so $j-1 \geq i$, hence $k_{j} \overline{\alpha'}[\gamma] < d_{j}  \overline{\alpha'} [\delta]$. Because $S(\overline{\alpha'})-j = l+1 >0$, we have $S(\overline{\alpha'}) > j$. Hence, $d_{j}\overline{\alpha'}[\gamma]  = \vartheta_{j}d_{j+1} \overline{\alpha'}[\gamma]$. The claim follows if $k_{j} \overline{\alpha'}[\gamma] < d_{j} \overline{\alpha'}[\delta]$ and $d_{j+1} \overline{\alpha'}[\gamma] < d_{j+1} \overline{\alpha'}[\delta]$, but we already know that both inequalities are true.

\medskip

Assertion $3.$: If $i<k$, then $k_{k+1} \overline{\alpha}[\gamma] = \overline{\alpha}[\gamma] < d_{k+1}\overline{\beta}[\delta]$ because $S(\overline{\alpha}[\gamma]) = i+1 < k+1$.\\
If $i>k$, then $k_{k+1} \overline{\alpha}[\gamma] = k_{k+1} \overline{\alpha'}[\gamma]$. Therefore, $G_k(\alpha) = G_k(\alpha') \cup \{\alpha'\} < \beta$ and the induction hypothesis yield the assertion.\\
Assume that $i=k$. Then $k_{k+1} \overline{\alpha}[\gamma] = \overline{\alpha}[\gamma] = d_{k+1} \overline{\alpha'}[\gamma]$ and $G_k(\alpha) = G_k(\alpha') \cup \{\alpha'\} < \beta$. The induction hypothesis on assertion $4.$ yields $d_{k+1} \overline{\alpha'}[\gamma] < d_{k+1} \overline{\beta}[\delta]$, from which we can conclude the assertion.

\medskip

Assertion $4.$: $\alpha < \beta$ yields $\beta = \pi_j \beta'$ with $i \leq j$.

If $i+1 = S( \overline{\alpha}) \leq k+1$, then $d_{k+1} \overline{\alpha}[\gamma] = \vartheta_{k+1} \overline{\alpha}[\gamma]$. There are two sub-cases: either $j+1 = \overline{\beta}[\delta] \leq k+1$ or not. In the former case, we obtain $d_{k+1} \overline{\beta}[\delta] = \vartheta_{k+1} \overline{\beta}[\delta]$. Assertion $4.$ then follows from assertions $1.$ and $3.$ and the induction hypothesis. In the latter case, we have $d_{k+1} \overline{\beta}[\delta]  =  \vartheta_{k+1} d_{k+2} \overline{\beta}[\delta]$. Assertion $4.$ follows from $\overline{\alpha}[\gamma] < d_{k+2} \overline{\beta}[\delta]$ and assertion $3.$ The previous strict inequality is valid because $S(\overline{\alpha}[\gamma]) = i+1 \leq k+1 < k+2$.

From now on assume that $i+1 = S( \overline{\alpha}) > k+1$. Actually, we only assume that $S(\overline{\alpha}) \geq k$.

$G_k \alpha < \beta$ yields $G_l \alpha < \beta$ for all $l \geq k$. We claim that $d_{j+1} \overline{\alpha} [\gamma] < d_{j+1} \overline{\beta} [\delta]$ for all $j \in \{k , \dots, S(\overline{\alpha})\}$ and show this by induction on $l= S(\overline{\alpha}) - j \in \{0, \dots, S(\overline{\alpha}) - k\}$. The assertion then follows from taking $l = S(\overline{\alpha}) - k$.

If $l=0$ or $l=1$, then $S(\overline{\alpha}) = k$ or equals $k+1$, hence the claim follows from the case $S( \overline{\alpha}) \leq k+1$. Assume that the claim is true for $l \geq 1$. We want to prove that this is also true for $l+1 =  S(\overline{\alpha}) - j $. The induction hypothesis on $l= S(\overline{\alpha}) - (j+1)$ yields $d_{j+2} \overline{\alpha} [\gamma] < d_{j+2} \overline{\beta} [\delta]$. Now because $l \geq 1$, we have $S(\overline{\beta}) \geq S(\overline{\alpha}) \geq j+2 > j+1$. So, $d_{j+1} \overline{\alpha} [\gamma] = \vartheta_{j+1}d_{j+2} \overline{\alpha} [\gamma]$ and $d_{j+1} \overline{\beta} [\delta] = \vartheta_{j+1} d_{j+2} \overline{\beta} [\delta]$. Then the claim is valid if $d_{j+2} \overline{\alpha} [\gamma] < d_{j+2} \overline{\beta} [\delta]$ and $k_{j+1} \overline{\alpha} [\gamma] < d_{j+1} \overline{\beta} [\delta]$. We already know the first strict inequality. The second one follows from assertion $3.$ and $j \geq k$.

\medskip

Assertion $5.$: We prove this by induction on $lh(\zeta) + lh(\eta)$. Assume $\zeta = \varphi_{\pi_0 \alpha} \gamma < \varphi_{\pi_0 \beta} \delta = \eta$. There are three cases.

Case 1: $\pi_0 \alpha < \pi_0 \beta$ and $\gamma < \eta$. The induction hypothesis yields $\psi(\gamma) < \psi(\eta)$. Furthermore, we know that $\alpha < \beta$. If $\alpha = 0$, then $d_0\overline{\alpha}[\psi(\gamma)]  = \vt_0 \psi(\gamma)$. We want to check if this is strictly smaller than $\psi(\eta) = d_0\overline{\beta}[\psi(\delta)] =\vartheta_0 d_1 \overline{\beta}[\psi(\delta)]$. Trivially $\psi(\gamma) <d_1 \overline{\beta}[\psi(\delta)]$. Furthermore, $k_0(\psi(\gamma)) = \psi(\gamma) < \psi(\eta)$. Hence $\psi(\zeta) = \vt_0 \psi(\gamma) < \vartheta_0 d_1 \overline{\beta}[\psi(\delta)] = \psi(\eta)$. Assume now $0<\alpha<\beta$. We want to prove that
\begin{align*}
&d_0\overline{\alpha}[\psi(\gamma)] = \vartheta_0 d_1 \overline{\alpha}[\psi(\gamma)] \\
<{}& {}d_0\overline{\beta}[\psi(\delta)] =\vartheta_0 d_1 \overline{\beta}[\psi(\delta)].
\end{align*}
Assertion $4.$, $\alpha < \beta$ and $G_0 (\alpha) < \alpha < \beta$ yield $d_1\overline{\alpha}[\psi(\gamma)] < d_1\overline{\beta}[\psi(\delta)]$. Additionally,
\[k_0 d_1\overline{\alpha}[\psi(\gamma)]  = \psi(\gamma) < \psi(\eta)  = \vartheta_0 d_1 \overline{\beta}[\psi(\delta)],\]
hence
$d_0\overline{\alpha}[\psi(\gamma)] < d_0\overline{\beta}[\psi(\delta)]$.

Case 2: $\pi_0 \alpha = \pi_0 \beta$ and $\gamma < \delta$. The induction hypothesis yields $\psi(\gamma) < \psi(\delta)$. Assertion $2.$ on $\pi_0 \alpha$ then yields $\overline{\pi_0 \alpha} [\psi(\gamma)] < \overline{\pi_0 \alpha} [\psi(\delta)]$. Hence, $d_1 \overline{\alpha} [\psi(\gamma)] < d_1 \overline{\alpha} [\psi(\delta)] =d_1 \overline{\beta} [\psi(\delta)]$.
Additionally,
\[k_0 d_1\overline{\alpha}[\psi(\gamma)]  = \psi(\gamma) < \psi(\delta) = k_0(d_1 \overline{\beta}[\psi(\delta)]) \leq \vartheta_0(d_1 \overline{\beta}[\psi(\delta)]),\]
hence $d_0\overline{\alpha}[\psi(\gamma)] < d_0\overline{\beta}[\psi(\delta)]$.

Case 3.: $\pi_0 \alpha > \pi_0 \beta$ and $\zeta < \delta$. Then $\psi(\zeta) < \psi( \delta) \leq k_0(d_1 \overline{\beta}[\psi(\delta)]) \leq \vartheta_0(d_1 \overline{\beta}[\psi(\delta)]) = \psi(\eta)$.
\end{proof}

\begin{corollary} $\om_{n+2}\leq \vtn\vte\ldots\vt_n\Om_{n+1}$
\end{corollary}
\begin{proof}
From the Theorems \ref{veblenhierarchy} and \ref{ordertypepi}, we know that the order type of $\varphi_{\pi_0(n)}0$ is $\omega_{n+2}$. Therefore, using assertion 5 in Lemma \ref{lemma:psiisorderpreserving}, we obtain $\omega_{n+2} \leq otype(T_{n+1}[0]) = \vt_0 \dots \vt_n \Omega_{n+1}$.
\end{proof}


\subsection{Upper bound}
In this subsection, we prove $\vt_0\vt_1\vt_2\ldots\vt_n\Omega_{n+1}  = otype(T_{n+1}[0]) \leq \om_{n+2}$. For this purpose, we introduce a new notation system with the same order type as $T_n$.

\begin{definition}
Let $n<\omega$. Define $T'_{n+1}$ as the least subset of $T_{n+1}$ such that
\begin{itemize}[noitemsep]
\item $0 \in T'_{n+1}$,
\item if $\alpha \in T'_{n+1}$, $S\alpha=i+1$ and $i<n$, then $\vartheta_i \alpha \in T'_{n+1}$,
\item if $\alpha \in T'_{n+1}$, then $\vartheta_n \alpha \in T'_{n+1}$.
\end{itemize}
Note that for all $\alpha \in T'_{n+1}$, we have $S\alpha \leq n$. Let $T'_0$ be $\{0\}$ and define $T'_n[m]$ accordingly as $T_n[m]$.
\end{definition}

\begin{lemma}\label{ordertypeT'_n}
The order types of $T'_{n}$ and $T_n$ are equal.
\end{lemma}
\begin{proof}
Trivially, $T'_n \subseteq T_n$, hence $otype(T'_n) \leq otype(T_n)$. Now, we give an order-preserving function $\psi$ from $T_n$ to $T'_n$. If $n=0$, this function appears trivially. So assume $n=m+1>0$.
\[
\begin{array}{llll}
 \psi: & T_{m+1} & \to &T'_{m+1},\\
 & 0 &\mapsto & 0,\\
 & \vartheta_i \alpha& \mapsto & \vartheta_i \vartheta_{i+1} \dots \vartheta_m \psi(\alpha).
\end{array}
\]

Let us first prove the following claim: for all $i\leq m$, if $\psi(\xi)  < \psi(\zeta) < \Omega_{i+1} = \vt_{i+1}0$, then $\psi(\vt_i \xi) < \psi(\vt_i \zeta)$. We prove this claim by induction on $m-i$. $i= m$, then $\psi(\vt_m \xi) = \vt_m \psi(\xi)$ and $\psi(\vt_m \zeta) = \vt_m \psi(\zeta)$. Hence, $\psi(\vt_m \xi) < \psi(\vt_m \zeta)$ easily follows because $k_m(\psi(\xi)) = \psi(\xi) < \psi(\zeta) = k_m(\psi(\zeta)) < \vt_m (\psi(\zeta))$.
Let $i<m$. Then
\begin{align*}
\psi(\vt_i \xi) = {}& {} \vt_i \dots \vt_m \psi(\xi),\\
\psi(\vt_i \zeta) = {} & {} \vt_i \dots \vt_m \psi(\zeta).
\end{align*}
Using the induction hypothesis, we obtain $\psi(\vt_{i+1} \xi) = \vt_{i+1} \dots \vt_m \psi(\xi) < \psi(\vt_{i+1} \zeta) = \vt_{i+1} \dots \vt_m \psi(\zeta)$. Furthermore, $k_i(\vt_{i+1} \dots \vt_m \psi(\xi)) = k_i(\psi(\xi)) = \psi(\xi) < \psi(\zeta) = k_i(\psi(\zeta)) = k_i(\vt_{i+1} \dots \vt_m \psi(\zeta)) < \vt_i(\vt_{i+1} \dots \vt_m \psi(\zeta)) $. Hence, $\psi(\vt_i \xi) $ $ = \vt_i \dots \vt_m \psi(\xi) < \psi(\vt_i \zeta) = \vt_i \dots \vt_m \psi(\zeta)$. This finishes the proof of the claim.

\medskip

Now we prove by main induction on $lh(\alpha) + lh(\beta)$ that $\alpha< \beta$ yields $\psi(\alpha) < \psi(\beta)$.
If $\alpha =0$, then the claim trivially holds.
Assume $0 < \alpha < \beta$. Then $\alpha = \vartheta_i \alpha'$ and $\beta = \vartheta_j \beta'$. If $i<j$, then $\psi(\alpha) < \psi(\beta)$ is also trivial. Assume $i=j\leq m $ and let $\alpha' = \vartheta_{j_1} \dots \vartheta_{j_k} k_i\alpha'$ and $\beta' = \vt_{n_1}\dots \vt_{n_l} k_i \beta'$ with $j_1,\dots,j_k,n_1,\dots, n_l >i$.
$\alpha < \beta$ either yields $\alpha \leq k_i \beta'$ or $\alpha' < \beta'$ and $k_i\alpha' < \beta$. In the former case, the induction hypothesis yields $\psi(\alpha) \leq \psi(k_i\beta') = k_i(\psi(\vt_{n_1}\dots \vt_{n_l} k_i \beta')) = k_i(\psi(\beta')) = k_i(\vartheta_{i+1} \dots \vartheta_m \psi(\beta'))$ \\$ < \vartheta_i(\vartheta_{i+1} \dots \vartheta_m \psi(\beta')) = \psi(\beta)$.

\medskip

Assume that we are in the latter case, meaning $\alpha' < \beta'$ and $k_i\alpha' < \beta$. The induction hypothesis yields $\psi\alpha' < \psi\beta'$ and $\psi(k_i\alpha') < \psi\beta$.
Like before, we attain $\psi(k_i \alpha') = k_i(\vt_{i+1} \dots \vt_m \psi(\alpha')) < \psi \beta = \vartheta_i (\vt_{i+1} \dots \vt_m \psi(\beta') )$. So if we can prove $\vt_{i+1} \dots \vt_m \psi(\alpha')  < \vt_{i+1} \dots \vt_m \psi(\beta') $, we are done. But this follows from the claim: if $i=j<m$, then $S(\alpha'), S(\beta') \leq i+1 \leq m$, hence $\psi(\alpha') < \psi(\beta')  < \Omega_{i+2}$, so $\vt_{i+1} \dots \vt_m \psi(\alpha')  = \psi(\vt_{i+1} \alpha') < \psi(\vt_{i+1}\beta') = \vt_{i+1} \dots \vt_m \psi(\beta')$. If $i=j=m$, then $\vt_{i+1} \dots \vt_m \psi(\alpha')$ and $\vt_{i+1} \dots \vt_m \psi(\beta') $ are actually $ \psi(\alpha')$ and $\psi(\beta') $ and we know that  $\psi(\alpha') <\psi(\beta') $ holds.
\end{proof}

The previous proof also yields that the order types of $T'_n[m]$ and $T_n[m]$ are equal.

\subsubsection{The instructive part: $\vt_0\vt_1\vt_2\Omega_{3} \leq  \omega^{\omega^{\omega^\omega}}$}\label{instructivepart}

In this subsection, we prove that $\omega^{\omega^{\omega^\omega}}$ is an upper bound for $\vt_0\vt_1\vt_2\Omega_{3}$ as an instructive instance for the general case
\[ \vt_0\vt_1\vt_2\ldots\vt_n\Omega_{n+1}  = otype(T_{n+1}[0]) \leq \om_{n+2}.\]
 We will show this by proving that $otype(T'_{3}[0]) \leq  \omega^{\omega^{\omega^\omega}}$.
We start with two simple lemmata, where we interpret $\Omega_i$ as usual as the $i^{th}$ uncountable cardinal number for $i>0$.

\begin{lemma}
If
$\Om_2\cdot \al+\be<\Om_2\cdot \ga+\de$ and $\al,\ga<\varepsilon_0$ and $\be,\de<\Om_2$
and if $\be=\xi\cdot \bes$ where $\bes<\Om_1\cdot \om^\ga+\om^{\om^\ga}\cdot \delta$ and $\xi<\om^{\om^\ga}$,
 then
$\Om_1\cdot \om^\al+\om^{\om^\al}\cdot \be<\Om_1\cdot \om^\ga+\om^{\om^\ga}\cdot \delta$.
\end{lemma}
\begin{proof}
Note that it is possible that $\beta, \delta \geq \Omega_1$.
 If $\al=\ga$ then $\be<\de$ and the assertion is obvious.
So assume $\al<\ga$.
$\bes<\Om_1\cdot \om^\ga+\om^{\om^\ga}\cdot \delta$ yields $\be=\xi \bes<\xi(\Om_1\cdot \om^\ga+\om^{\om^\ga}\cdot \delta)= \Om_1\cdot \om^\ga+\om^{\om^\ga}\cdot \delta$ since $\Om_1$ and $\om^{\om^\ga}$ are multiplicatively closed.
By the same argument $\om^{\om^\al}\be<\om^{\om^\al} (\Om_1\cdot \om^\ga+\om^{\om^\ga}\cdot \delta)= \Om_1\cdot \om^\ga+\om^{\om^\ga}\cdot \delta$.
Finally, $\Om_1\cdot \om^\al+\om^{\om^\al}\cdot \be<  \Om_1\cdot \om^\al+\Om_1\cdot \om^\ga+\om^{\om^\ga}\cdot \delta=\Om_1\cdot \om^\ga+\om^{\om^\ga}\cdot \delta.$
\end{proof}

\begin{lemma}
If $\Om_1\cdot \al+\be<\Om_1\cdot \ga+\de$ and $\al,\ga<\varepsilon_0$ and $\be,\de<\Om_1$ and if  $\be<\om^{\om^\ga}\cdot \delta$, then $\om^{\om^\al}\cdot \be<\om^{\om^\ga}\cdot \delta$.
\end{lemma}
\begin{proof}
If $\alpha = \gamma$, then $\be < \de$ and the assertion is obvious. So assume $\al <  \ga$. Then $\om^{\om^\al} \cdot \be <\om^{\om^\al} \om^{\om^\ga}\cdot \delta  = \om^{\om^\ga}\cdot \delta $.
\end{proof}

\medskip

The last two lemmas indicate how one might replace iteratively terms in $\vt_i$ (starting with the highest level $i$) by terms in $\om,+,\Om_i$ in an order-preserving way such that terms of level $0$ are smaller than $\epsilon_0$.

\begin{definition}
Define $E$ as the least set such that
\begin{itemize}[noitemsep]
\item $0 \in E$,
\item $\alpha \in E$, then $\omega^\alpha \in E$,
\item $\alpha, \beta \in E$, then $\alpha + \beta \in E$.
\end{itemize}
Define the subset $P$ of $E$ as the set of all elements of the form $\omega^\alpha$ for $\alpha \in E$. This actually means that $P$ is the set of the additively closed ordinals strictly below $\varepsilon_0$.
\end{definition}

A crucial role is played by the following function $f$.
\begin{definition} Let $f(0):=0$ and $f(\om^\ale+\alz):=\om^\ale+f(\ale)+f(\alz)$.
\end{definition}

This definition even works (by magic) also for non Cantor normal forms. So if $\om^\ale+\alz=\alz$ we still have $f(\om^\ale+\alz)=\om^\ale+f(\ale)+f(\alz)[=f(\alz)]$. The function $f$ is easily shown to be order-preserving. Moreover, one finds $\om^\ale\leq f(\om^\ale+\alz)<\om^{\ale+1}$ if $\alpha_2 < \omega^{\alpha_1 + 1}$.

Fix a natural number  $n$.
We formally work with $4$-tuples $(\alpha,\beta,\gamma,\delta) \in E \times T[n-1] \times P \times E$ with $\alpha,\delta \in E$, $\gamma \in P$, $\beta \in T[n-1]$ and $\delta < \gamma$. Let $T[-1] := \{0\}$.
We order these tuples lexicographically. 
Intuitively, we interpret such a tuple as the ordinal
\[\Omega_n \cdot \alpha + \gamma \cdot \beta + \delta,\]
where $\Omega_i$ is as usual the $i^{th}$ uncountable ordinal for $i>0$, but now $\Omega_0$ is interpreted as $0$.

\medskip

We remark that the interpretation of $(\alpha,\beta,\gamma,\delta)$ as an ordinal number is not entirely correct: the lexicographic order on the tuples is not the same as the induced order by the ordering on the class of ordinals $On$. But in almost all applications, we know that $\gamma = \omega^{f(\alpha)}$. And if this true, we know that the order induced by the ordering on $On$ is the same as the defined lexicographic one.
Additionally, the encountered cases where $\gamma \neq \omega^{f(\alpha)}$, we know that if we compare two tuples $(\alpha,\beta,\gamma,\delta)$ and $(\alpha',\beta',\gamma',\delta')$ such that $\alpha = \alpha'$, then we already know that $\gamma = \gamma'$. Hence, the order induced by the ordering on $On$ between these terms is also the same as the lexicographic one.

$\beta$ is either $0$ or of the form $\vt_j \be'$ with $j<n$, hence we can interpret that $\beta < \Omega_n$ for $n>0$. Assume that $\zeta \in P$. Then we know that $\zeta \cdot \Omega_n = \Omega_n$. Hence using all of these interpretations, $\zeta \cdot (\alpha,\beta,\gamma,\delta)$ is still a 4-tuple, namely it is equal to $(\alpha,\beta,\zeta \cdot \gamma,\zeta \cdot \delta)$. We can also define the sum between 4-tuples: assume $n>0$. If $\alpha' >0$, then
\begin{eqnarray*}
(\alpha,\beta,\gamma,\delta) + (\alpha',\beta',\gamma',\delta') & =& \Omega_n \cdot \alpha +  \gamma \cdot \beta + \delta  + \Omega_n \cdot \alpha' + \gamma' \cdot \beta' + \delta' \\
&=& \Omega_n \cdot (\alpha + \alpha') + \gamma' \cdot \beta' + \delta'\\
&=& (\al + \alpha',\beta',\gamma',\delta')
\end{eqnarray*}
If $\alpha' =0$ and $\beta' = 0$, then
\begin{eqnarray*}
(\alpha,\beta,\gamma,\delta) + (\alpha',\beta',\gamma',\delta') & =& \Omega_n \cdot \alpha +  \gamma \cdot \beta + \delta  + \Omega_n \cdot \alpha' + \gamma' \cdot \beta' + \delta' \\
&=& \Omega_n \alpha  + \gamma \cdot \beta + (\delta + \delta')\\
&=& (\al,\beta,\gamma,\delta + \delta')
\end{eqnarray*}
We do not need the case $\al' =0$ and $\be'\neq 0$. If $n=0$, then
\begin{eqnarray*}
(\alpha,\beta,\gamma,\delta) + (\alpha',\beta',\gamma',\delta') & =& \Omega_n \cdot \alpha +  \gamma \cdot \beta + \delta  + \Omega_n \cdot \alpha' + \gamma' \cdot \beta' + \delta' \\
& =&  \delta  + \delta' \\
&=& (0,0,0, \delta + \delta')
\end{eqnarray*}

From now on, we write
\[\Omega_n \cdot \alpha + \gamma \cdot \beta + \delta,\]
instead of the 4-tuple $(\alpha,\beta,\gamma,\delta)$, although we know that the induced order by the ordering on $On$ is not entirely the same as the lexicographic one.

\begin{definition}
Define $T_n^{all}$ as the set consisting of $\Omega_n \cdot \alpha + \omega^{f(\alpha)} \cdot \delta + \gamma$, where $\alpha,\gamma \in E$ with $\gamma < \omega^{f(\alpha)}$ and $\delta \in T[n-1]$. 
\end{definition}

Note that after an obvious translation, $T_0^{all} = E$ and $T_n \subseteq T[n-1] \subseteq T_n^{all}$. 
%


\begin{lemma}\label{vt1vt2} Assume $\als,\bes \in T[0]$.
  If
  \[\al=\vte \vtz^{n_1}\ldots\vte \vtz^{n_p}\als<\be=\vte \vtz^{l_1}\ldots\vte \vtz^{l_q}\bes\]
  with $n_i,l_i>0$, then
\begin{align*}
  & \Om_1\cdot (\om^{n_1}+\cdots +\om^{n_p})+\om^{\om^{n_1}+\cdots +\om^{n_p}+ n_p}\cdot \als+\om^{\om^{n_1}+\cdots +\om^{n_{p}}}  \\
  & +\om^{\om^{n_1}+\cdots +\om^{n_{p-1}}}+\cdots +\om^{\om^{n_1}}\\
  < {}& {} \Om_1\cdot (\om^{l_1}+\cdots +\om^{l_q})+\om^{\om^{l_1}+\cdots +\om^{l_q}+l_q}\cdot \bes+\om^{\om^{l_1}+\cdots +\om^{l_{q}}} \\
 & +\om^{\om^{l_1}+\cdots +\om^{l_{q-1}}}+\cdots + \omega^{\omega^{l_1}}.
  \end{align*}
  \end{lemma}
  \begin{proof}
Note that $f(\om^{n_1}+\cdots +\om^{n_p}) = \om^{n_1}+\cdots +\om^{n_p} + n_p$ and that $\om^{n_1}+\cdots +\om^{n_p}$ is not necessarily in Cantor normal form.
 We prove by induction on $lh(\alpha)-lh(\alpha') + lh(\beta) - lh(\beta')$ that the assumption yields
 \begin{align*}
 & ( \om^{n_1}+\cdots +\om^{n_p} , \alpha' , \om^{n_1}+\cdots +\om^{n_{p-1}} ,\dots , \om^{n_1} ) \\
 <_{lex}{} &{} (\om^{l_1}+\cdots +\om^{l_q}, \beta' , \om^{l_1}+\cdots +\om^{l_{q-1}},\dots , \om^{l_1}).
 \end{align*}

From this inequality, the lemma follows.

\medskip

If $lh(\alpha) = lh(\alpha')$, then $p=0$. If $q>0$, then this is trivial, so we can assume that $q$ is also $0$. But then $\om^{n_1}+\cdots +\om^{n_p}  = \om^{l_1}+\cdots +\om^{l_q} = 0$ and $\alpha' = \alpha < \beta = \beta'$.
Now assume that $p>0$. It is impossible that $q=0$. $\alpha < \beta$ yields either
$\vte \vtz^{n_1}\ldots\vte \vtz^{n_p}\als<\vte \vtz^{l_2}\ldots\vte \vtz^{l_q} \be'$
or
$ \vtz^{n_1}\ldots\vte \vtz^{n_p}\als< \vtz^{l_1}\ldots\vte \vtz^{l_q}\bes$
and $\vte \vtz^{n_2}\ldots\vte \vtz^{n_p}\als<\vte \vtz^{l_1}\ldots\vte \vtz^{l_q}\bes$.

\medskip

In the former case, the induction hypothesis yields
 \begin{align*}
 & ( \om^{n_1}+\cdots +\om^{n_p} , \alpha' , \om^{n_1}+\cdots +\om^{n_{p-1}} ,\dots , \om^{n_1} ) \\
 <_{lex}{} &{} (\om^{l_2}+\cdots +\om^{l_q}, \beta' , \om^{l_2}+\cdots +\om^{l_{q-1}},\dots , \om^{l_2}).
  \end{align*}
If $l_2 \leq l_1$, then trivially
 \begin{align*}
& {} (\om^{l_2}+\cdots +\om^{l_q}, \beta' , \om^{l_2}+\cdots +\om^{l_{q-1}},\dots ,\omega^{l_2} , \om^{l_1})\\
  <_{lex}{} &{} (\om^{l_1}+\cdots +\om^{l_q}, \beta' , \om^{l_1}+\cdots +\om^{l_{q-1}},\dots , \omega^{l_1} + \omega^{l_2} , \om^{l_1}).
 \end{align*}
  If $l_2 > l_1$, then
  \begin{align*}
  & (\om^{l_2}+\cdots +\om^{l_q}, \beta' , \om^{l_2}+\cdots +\om^{l_{q-1}},\dots , \om^{l_2})\\
 ={} & {}(\om^{l_1}+ \omega^{l_2}+\cdots +\om^{l_q}, \beta' , \om^{l_1} +  \om^{l_2}+\cdots +\om^{l_{q-1}},\dots , \om^{l_1}+ \om^{l_2})\\
<_{lex} {} & {} (\om^{l_1}+\cdots +\om^{l_q}, \beta' , \om^{l_1}+\cdots +\om^{l_{q-1}},\dots , \omega^{l_1} + \omega^{l_2} , \om^{l_1}).
 \end{align*}
\medskip

Assume that we are in the latter case. $ \vtz^{n_1}\ldots\vte \vtz^{n_p}\als< \vtz^{l_1}\ldots\vte \vtz^{l_q}\bes$ yields $n_1<l_1$ or $n_1 = l_1$ and $\vte \vtz^{n_2}\ldots\vte \vtz^{n_p}\als<\vte \vtz^{l_2}\ldots\vte \vtz^{l_q}\bes$.

\medskip

Suppose $n_1 < l_1$. The induction hypothesis on
\[\vte \vtz^{n_2}\ldots\vte \vtz^{n_p}\als<\vte \vtz^{l_1}\ldots\vte \vtz^{l_q}\bes\]
implies
 \begin{align*}
 & ( \om^{n_2}+\cdots +\om^{n_p} , \alpha' , \om^{n_2}+\cdots +\om^{n_{p-1}} ,\dots , \om^{n_2} ) \\
 <_{lex}{} &{} (\om^{l_1}+\cdots +\om^{l_q}, \beta' , \om^{l_1}+\cdots +\om^{l_{q-1}},\dots , \om^{l_1}).
 \end{align*}
 Let
 \begin{align*}
 s := {} & {}  ( \om^{n_2}+\cdots +\om^{n_p} , \alpha' , \om^{n_2}+\cdots +\om^{n_{p-1}} ,\dots , \om^{n_2} ) \\
s' := {} &{} (\om^{l_1}+\cdots +\om^{l_q}, \beta' , \om^{l_1}+\cdots +\om^{l_{q-1}},\dots , \om^{l_1}).
 \end{align*}
Note that $lh(s) = p$ and $lh(s') = q+1$. If $lh(s) < lh(s')$ and $s_i = s'_i$ for all $i<lh(s)$, then
  \begin{align*}
 & ( \om^{n_1}+\cdots +\om^{n_p} , \alpha' , \om^{n_1}+\cdots +\om^{n_{p-1}} ,\dots , \om^{n_1} ) \\
= {}  & {} ( \om^{n_2}+\cdots +\om^{n_p} , \alpha' , \om^{n_2}+\cdots +\om^{n_{p-1}} ,\dots , \om^{n_2}, \om^{n_1} ) \\
 <_{lex}{} &{} (\om^{l_1}+\cdots +\om^{l_q}, \beta' , \om^{l_1}+\cdots +\om^{l_{q-1}},\dots , \om^{l_1}),
 \end{align*}
 where for the last inequality we need $n_1 < l_1$ if $p=q$. If there exists an index $j< \min\{lh(s),lh(s')\}$ such that $s_j < s'_j$ and $s_i = s'_i$ for all $i<j$, then
  \begin{align*}
   & ( \om^{n_1}+\cdots +\om^{n_p} , \alpha' , \om^{n_1}+\cdots +\om^{n_{p-1}} ,\dots , \om^{n_1} ) \\
 & ( \om^{n_2}+\cdots +\om^{n_p} , \alpha' , \om^{n_2}+\cdots +\om^{n_{p-1}} ,\dots , \om^{n_2}, \om^{n_1} ) \\
 <_{lex}{} &{} (\om^{l_1}+\cdots +\om^{l_q}, \beta' , \om^{l_1}+\cdots +\om^{l_{q-1}},\dots , \om^{l_1}).
 \end{align*}

 Now assume $n_1 = l_1$. The induction hypothesis on $\vte \vtz^{n_2}\ldots\vte \vtz^{n_p}\als<\vte \vtz^{l_2}\ldots\vte \vtz^{l_q}\bes$ implies
   \begin{align*}
 & ( \om^{n_2}+\cdots +\om^{n_p} , \alpha' , \om^{n_2}+\cdots +\om^{n_{p-1}} ,\dots , \om^{n_2} ) \\
 <_{lex}{} &{} (\om^{l_2}+\cdots +\om^{l_q}, \beta' , \om^{l_2}+\cdots +\om^{l_{q-1}},\dots , \om^{l_2}).
 \end{align*}
 Let
   \begin{align*}
s := {} & {} ( \om^{n_2}+\cdots +\om^{n_p} , \alpha' , \om^{n_2}+\cdots +\om^{n_{p-1}} ,\dots , \om^{n_2} ) \\
 s' := {} &{} (\om^{l_2}+\cdots +\om^{l_q}, \beta' , \om^{l_2}+\cdots +\om^{l_{q-1}},\dots , \om^{l_2}).
 \end{align*}
Note that $lh(s) = p$ and $lh(s') = q$. If $lh(s) < lh(s')$ and $s_i = s'_i$ for all $i<lh(s)$, then one can easily prove
  \begin{align*}
 & ( \om^{n_1}+\om^{n_2} + \cdots +\om^{n_p} , \alpha' , \om^{n_1}+\cdots +\om^{n_{p-1}} ,\dots , \om^{n_1} + \om^{n_2} ) \\
 <_{lex}{} &{} (\om^{l_1}+ \om^{l_2} +\cdots +\om^{l_q}, \beta' , \om^{l_1}+\cdots +\om^{l_{q-1}},\dots , \om^{l_1}+ \om^{l_2}),
 \end{align*}
hence
  \begin{align*}
 & ( \om^{n_1}+\om^{n_2} + \cdots +\om^{n_p} , \alpha' , \om^{n_1}+\cdots +\om^{n_{p-1}} ,\dots , \om^{n_1} + \om^{n_2} , \om^{n_1} ) \\
  <_{lex}{} &{} (\om^{l_1}+ \om^{l_2} +\cdots +\om^{l_q}, \beta' , \om^{l_1}+\cdots +\om^{l_{q-1}},\dots , \om^{l_1}+ \om^{l_2})\\
 <_{lex}{} &{} (\om^{l_1}+ \om^{l_2} +\cdots +\om^{l_q}, \beta' , \om^{l_1}+\cdots +\om^{l_{q-1}},\dots , \om^{l_1}+ \om^{l_2}, \om^{l_1}).
 \end{align*}

 If there exists an index $j< \min\{lh(s),lh(s')\}$ such that $s_j < s'_j$ and $s_i = s'_i$ for all $i<j$, then also
 \begin{align*}
 & ( \om^{n_1}+\om^{n_2} + \cdots +\om^{n_p} , \alpha' , \om^{n_1}+\cdots +\om^{n_{p-1}} ,\dots , \om^{n_1} + \om^{n_2} ) \\
 <_{lex}{} &{} (\om^{l_1}+ \om^{l_2} +\cdots +\om^{l_q}, \beta' , \om^{l_1}+\cdots +\om^{l_{q-1}},\dots , \om^{l_1}+ \om^{l_2}),
 \end{align*}
hence
  \begin{align*}
 & ( \om^{n_1}+\om^{n_2} + \cdots +\om^{n_p} , \alpha' , \om^{n_1}+\cdots +\om^{n_{p-1}} ,\dots , \om^{n_1} + \om^{n_2} , \om^{n_1} ) \\
 <_{lex}{} &{} (\om^{l_1}+ \om^{l_2} +\cdots +\om^{l_q}, \beta' , \om^{l_1}+\cdots +\om^{l_{q-1}},\dots , \om^{l_1}+ \om^{l_2}, \om^{l_1}).
 \end{align*}

 \end{proof}

 Define $\tau_0$ as the mapping from $T'_3[0]$ to $T^{all}_0 = E$ as follows: let $\tau_0 0 := 0$. If $\alpha = \vtn\vte\vtz^{n_1}\ldots\vte\vtz^{n_p}\als$ with $\als \in T'_3[0]$ and $n_1,\dots,n_p,p >0$, define $\tau_0 \al$ as
\[\om^{\om^{\om^{n_1}+\cdots +\om^{n_p}}}\cdot (\om^{\om^{n_1}+\cdots +\om^{n_p}+n_p}\cdot \tau_0\als+\om^{\om^{n_1}+\cdots +\om^{n_{p}}}+\om^{\om^{n_1}+\cdots +\om^{n_{p-1}}}+\cdots +\om^{\om^{n_1}}).\]

\begin{lemma} Assume $\al,\be \in T'_3[0]$. If $\alpha < \beta$, then $\tau_0 \alpha < \tau_0 \beta$.

%
%
%
%
%
  \end{lemma}
 \begin{proof}
  We prove this by induction on the length of $\alpha$ and $\beta$. If $\alpha =0$, then this is trivial. So we can assume that $0<\alpha <\beta$. Hence,
  \[\alpha = \vtn\vte\vtz^{n_1}\ldots\vte\vtz^{n_p}\als \]
  and
  \[ \be=\vtn\vte\vtz^{l_1}\ldots\vte\vtz^{l_q}\bes \]
   with $\als,\bes \in T'_3[0]$ and $n_1,\dots,n_p,l_1,\dots,l_q,p,q >0$.

  We want to prove that
  \begin{align*}
 & \tau_0\al=\om^{\om^{\om^{n_1}+\cdots +\om^{n_p}}}\cdot (\om^{\om^{n_1}+\cdots +\om^{n_p}+n_p}\cdot \tau_0\als+\om^{\om^{n_1}+\cdots +\om^{n_{p}}} +\cdots +\om^{\om^{n_1}})\\
<{} &{}  \tau_0\be= \om^{\om^{\om^{l_1}+\cdots +\om^{l_q}}}\cdot(\om^{\om^{l_1}+\cdots +\om^{l_q}+l_q}\cdot \tau_0\bes+\om^{\om^{l_1}+\cdots +\om^{l_{q}}} +\cdots +\om^{\om^{l_1}}).
 \end{align*}
$\al=\vtn\vte\vtz^{n_1}\ldots\vte\vtz^{n_p}\al'<\be=\vtn\vte\vtz^{l_1}\ldots\vte\vtz^{l_q}\be'$ yields two cases: either $\al \leq k_0(\vte\vtz^{l_1}\ldots\vte\vtz^{l_q}\be') =\be'$ or $\vte\vtz^{n_1}\ldots\vte\vtz^{n_p}\al'<\vte\vtz^{l_1}\ldots\vte\vtz^{l_q}\be'$ and $\al'<\be$. In the former case, the induction hypothesis yields $\tau_0 \alpha  \leq \tau_0 \beta' < \tau_0 \beta$.

\medskip

So assume the latter case. Then the induction hypothesis yields $\tau_0 \alpha' < \tau_0 \beta$. Using Lemma \ref{vt1vt2}, we know that
\begin{align*}
  & \Om_1\cdot (\om^{n_1}+\cdots +\om^{n_p})+\om^{\om^{n_1}+\cdots +\om^{n_p}+ n_p}\cdot \tau_0\als+\om^{\om^{n_1}+\cdots +\om^{n_{p}}}  \\
  & +\om^{\om^{n_1}+\cdots +\om^{n_{p-1}}}+\cdots +\om^{\om^{n_1}}\\
  < {}& {} \Om_1\cdot (\om^{l_1}+\cdots +\om^{l_q})+\om^{\om^{l_1}+\cdots +\om^{l_q}+l_q}\cdot \tau_0\bes+\om^{\om^{l_1}+\cdots +\om^{l_{q}}} \\
 & +\om^{\om^{l_1}+\cdots +\om^{l_{q-1}}}+\cdots + \omega^{\omega^{l_1}}.
  \end{align*}

If $\om^{n_1}+\cdots +\om^{n_p} < \om^{l_1}+\cdots +\om^{l_q}$, then
\[\om^{\om^{\om^{n_1}+\cdots +\om^{n_p}}}\cdot \om^{\om^{n_1}+\cdots +\om^{n_p}+n_p}\cdot \tau_0\als < \om^{\om^{\om^{n_1}+\cdots +\om^{n_p}}}\cdot \om^{\om^{n_1}+\cdots +\om^{n_p}+n_p}\tau_0 \beta = \tau_0 \beta.\]
 Therefore,
 \begin{align*}
 &\om^{\om^{\om^{n_1}+\cdots +\om^{n_p}}}\cdot (\om^{\om^{n_1}+\cdots +\om^{n_p}+n_p}\cdot \tau_0\als+\om^{\om^{n_1}+\cdots +\om^{n_{p}}} +\cdots +\om^{\om^{n_1}})\\
 < {}&  \om^{\om^{\om^{n_1}+\cdots +\om^{n_p}}}\cdot \om^{\om^{n_1}+\cdots +\om^{n_p}+n_p}\cdot \tau_0\als\\
&   +
 \om^{\om^{\om^{n_1}+\cdots +\om^{n_p}}} \cdot (\om^{\om^{n_1}+\cdots +\om^{n_{p}}} +\cdots +\om^{\om^{n_1}})\\
 < {}& \tau_0 \beta,
\end{align*}
because $ \om^{\om^{\om^{n_1}+\cdots +\om^{n_p}}} \cdot (\om^{\om^{n_1}+\cdots +\om^{n_{p}}} +\cdots +\om^{\om^{n_1}}) < \om^{\om^{\om^{l_1}+\cdots +\om^{l_q}}}$. We used the standard observation that $\xi < \rho + \omega^\mu$ and $\lambda < \mu$ imply $\xi + \omega^\lambda < \rho  + \omega^\mu$.

\medskip

Assume $\om^{n_1}+\cdots +\om^{n_p} = \om^{l_1}+\cdots +\om^{l_q}$ and $\tau_0 \alpha' < \tau_0 \beta'$. Then $\tau_0 \alpha < \om^{\om^{\om^{n_1}+\cdots +\om^{n_p}}}\cdot \om^{\om^{n_1}+\cdots +\om^{n_p}+n_p}\cdot (\tau_0\als+1) \leq \om^{\om^{\om^{n_1}+\cdots +\om^{n_p}}}\cdot \om^{\om^{n_1}+\cdots +\om^{n_p}+n_p}\cdot \tau_0\bes \leq \tau_0 \beta$.

\medskip

Assume $\om^{n_1}+\cdots +\om^{n_p} = \om^{l_1}+\cdots +\om^{l_q}$, $\tau_0 \alpha' = \tau_0 \beta'$ and $\om^{\om^{n_1}+\cdots +\om^{n_{p}}} +\om^{\om^{n_1}+\cdots +\om^{n_{p-1}}}+\cdots +\om^{\om^{n_1}} < \om^{\om^{l_1}+\cdots +\om^{l_{q}}} +\om^{\om^{l_1}+\cdots +\om^{l_{q-1}}}+\cdots + \omega^{\omega^{l_1}}$. Then trivially, $\tau_0 \alpha < \tau_0 \beta$.

\end{proof}

\subsubsection{The general part: $\vt_0\dots \vt_n\Omega_{n+1} \leq  \omega_{n+2}$}

We show that $otype(T'_{n+1}[0]) \leq \omega_{n+2}$. The previous section give us the idea of how to deal with this question, however the order-preserving embeddings in this subsection are slightly different than the ones proposed in the previous Subsection \ref{instructivepart} for technical reasons. Fix  a natural number $n$ strictly bigger than $0$.

\begin{definition} $\tau_m$ are functions from $T'_{n+1}[m]$ to $T_{m}^{all}$. We define $\tau_m \alpha$ for all $m$ simultaneously by induction on the length of $\alpha$.
If $m \geq n+1$, then $T'_{n+1}[m] = T'_{n+1}$ and define $\tau_m \alpha = \alpha = \Omega_{m} 0 + \omega^0 \alpha + 0$ for all $\alpha$. Note that $\alpha \in T'_{n+1} \subseteq T[n] \subseteq T[m-1]$.
Assume $m \leq n$. Define $\tau_m 0$ as $0$. Define $\tau_m \vartheta_j \alpha$ as $\vartheta_j \alpha$ if $j<m$. Define $\tau_m \vartheta_m \alpha$ as $\Omega_m \omega^{\beta} + \omega^{\omega^{\beta}} ( \omega^{f(\beta)} \cdot \tau_m k_m \alpha + \eta)+1$ if $\tau_{m+1} \alpha = \Omega_{m+1} \beta + \omega^{f(\beta)} k_m \alpha + \eta$.
\end{definition}

First we prove that $\tau_m$ is well-defined.

\begin{lemma}\label{taumwelldefined}
For all $m> 0$ and $\alpha \in T'_{n+1}[m]$, there exist uniquely determined $\beta$ and $\eta$ with $\eta < \omega^{f(\beta)}$ such that $\tau_m \alpha = \Omega_m \beta + \omega^{f(\beta)} k_{m-1} \alpha + \eta$. 
Furthermore, $\eta$ is either zero or a successor.
\end{lemma}
\begin{proof}
We prove the first claim by induction on $lh(\alpha)$ and $n+1-m$. If $m\geq n+1$, then this is trivial by definition. Assume $0< m\leq n$. From the induction hypothesis, we know that there exist $\beta$, $\eta$, $\beta_1$, $\eta_1$ such that $\tau_{m+1} \alpha = \Omega_{m+1} \beta + \omega^{f(\beta)} k_m \alpha + \eta$ with $\eta < \omega^{f(\beta)}$ and $\tau_m k_m \alpha = \Omega_m \beta_1 + \omega^{f(\beta_1)} k_{m-1} k_m \alpha + \eta_1$ with $\eta_1 < \omega^{f(\beta_1)}$. We want to prove that there exist $\beta'$ and $\eta'$ such that $\tau_m \vt_m \alpha = \Omega_m \beta' + \omega^{f(\beta')} k_{m-1} \vt_m \alpha + \eta'$ with $\eta' < \omega^{f(\beta')}$.
Using the definition,
   \begin{eqnarray*}
   &&\tau_m \vartheta_m\al\\
   &=&\Omega_m \omega^{\beta} + \omega^{\omega^{\beta}} ( \omega^{f(\beta)} \cdot \tau_m k_m \alpha + \eta)+1\\
   &=&\Omega_m \omega^{\beta} + \omega^{\omega^{\beta}} ( \omega^{f(\beta)} \cdot  (\Omega_m \beta_1 + \omega^{f(\beta_1)} k_{m-1} k_m \alpha + \eta_1) + \eta)+1\\
   &=&\Omega_m (\omega^{\beta} + \beta_1) + \omega^{\omega^{\beta}}  \omega^{f(\beta)}(\omega^{f(\beta_1)} k_{m-1} k_m \alpha + \eta_1) + \omega^{\omega^{\beta}} \eta+1\\
   &=&\Omega_m (\omega^{\beta} + \beta_1) + \omega^{\omega^{\beta}}  \omega^{f(\beta)}\omega^{f(\beta_1)} k_{m-1} k_m \alpha +  \omega^{\omega^{\beta}}  \omega^{f(\beta)}\eta_1  + \omega^{\omega^{\beta}} \eta+1\\
 &=&\Omega_m (\omega^{\beta} + \beta_1) + \omega^{f(\omega^{\beta}+ \beta_1)}  k_{m-1} k_m \alpha +  \omega^{\omega^{\beta}}  \omega^{f(\beta)}\eta_1  + \omega^{\omega^{\beta}} \eta+1\\
  &=&\Omega_m (\omega^{\beta} + \beta_1) + \omega^{f(\omega^{\beta}+ \beta_1)}  k_{m-1} \vartheta_m \alpha +  \omega^{\omega^{\beta}}  \omega^{f(\beta)}\eta_1  + \omega^{\omega^{\beta}} \eta+1.
    \end{eqnarray*}

Define $\beta'$ as $\omega^\beta + \beta_1>0$ and $\eta'$ as $\omega^{\omega^{\beta}}  \omega^{f(\beta)}\eta_1  + \omega^{\omega^{\beta}} \eta+1$. Note that $\omega^{\omega^{\beta}}  \omega^{f(\beta)}\eta_1  < \omega^{\omega^{\beta}}  \omega^{f(\beta)} \omega^{f(\beta_1)}  = \omega^{f(\beta')}$, $\omega^{\omega^{\beta}} \eta < \omega^{\omega^{\beta} + f(\beta) }  \leq \omega^{f(\beta')}$ and $1 < \omega^{f(\beta')}$, hence $\eta' < \omega^{f(\beta')}$.

\medskip

That $\eta$ is either zero or a successor for all $m$ and $\alpha$ follows by construction.

\end{proof}

The argument in the proof of Lemma \ref{taumwelldefined} is crucially based on the property of $f$ regarding non-normal forms. 
The lemma implies that $\tau_m$ is well-defined for all $m>0$ and it does not make sense for $m=0$ because we did not define $k_{-1} \al$. But, looking to the definition of $\tau_0$, it is easy to see that $\tau_0$ is also well-defined.

\medskip

Note that one can easily prove $\tau_0 \alpha \in T^{all}_0$ for all $\alpha \in T'_{n+1}[0]$. Furthermore, $\tau_0 \alpha$ is also either zero or a successor ordinal. For all $m$ and $\alpha$, define $(\tau_m \alpha)^-$ as $\tau_m\alpha$, if $\eta$ is zero, and as $\tau_m \alpha$ but with $\eta-1$ instead of $\eta$, if $\eta$ is a successor. Additionally, note that if $m>0$ and $\tau_m \alpha = \Omega_m \beta + \omega^{f(\beta)} k_{m-1} \alpha + \eta$ we have $\beta >0$ iff $\eta >0$.

\medskip

 In the next theorem, we will again use the standard observation that $\xi < \rho + \omega^\mu$ and $\lambda < \mu$ imply $\xi + \omega^\lambda < \rho  + \omega^\mu$.

\begin{theorem}\label{tauiorderpreserving} For all natural $m$ and $\alpha, \beta \in T'_{n+1}[m]$, if $ \al<\be$, then $\tau_m\al<\tau_m\be$.
\end{theorem}
\begin{proof}
We prove this theorem by induction on $lh\al+lh\be$. If $\alpha$ and/or $\beta$ are zero, this is trivial. So we can assume that $\alpha = \vti\als$ and $\be = \vt_j \bes$.
One can easily prove the statement if $i<j$, even if $j=m$. So we can assume that $i=j$. If $i=j<m$, then this is also easily proved. So suppose that $i=j=m$. 
If $m>n$, then $\tau_m \alpha = \alpha < \beta = \tau_m \beta$, hence we are done. So we can also assume that $m\leq n$.

\medskip

$\alpha = \vt_m \als < \vt_m \bes$ yields $\alpha \leq k_m \bes$ or $\als < \bes$ and $k_m \als < \be$.
In the former case, the induction hypothesis yields $\tau_m \alpha  \leq \tau_m k_m \bes < \tau_m \vt_m \bes $ $ = \tau_m \be$, where $\tau_m k_m \bes < \tau_m \vt_m \bes$ follows from the definition of $\tau_m \vt_m \bes$. (One can also look at the proof of Lemma \ref{taumwelldefined} for $m>0$. The case $m=0$ is straightforward.)
So we only have to prove the assertion in the latter case, i.e. if $\als < \bes$ and $k_m \als < \be$. The induction hypothesis yields $\tau_{m+1} \als <\tau_{m+1} \bes$ and $\tau_m k_m \als < \tau_m \be$.
Assume
\begin{align*}
\tau_{m+1}\als=&\Om_{m+1}\cdot \ale+\om^{f(\ale)}\cdot k_m\als+\alz,\\
\tau_{m+1}\bes=&\Om_{m+1}\cdot \bee+\om^{f(\bee)}\cdot k_m\bes+\bez,
\end{align*}
where $\alz< \om^{f(\ale)}$, $\bez< \om^{f(\bee)}$.
Then
\begin{align*}
\tau_m\al=&\Om_m\cdot \om^{\ale}+\om^{\om^\ale}(\om^{f(\ale)}\cdot \tau_m k_m\als+\alz)+1,\\
\tau_m\be=&\Om_m\cdot \om^{\bee}+\om^{\om^\bee}(\om^{f(\bee)}\cdot \tau_m k_m\bes+\bez)+1.
\end{align*}
The inequality $\tau_{m+1}\als<\tau_{m+1}\bes$ yields $\ale\leq\bee$.
Assume first that $\ale=\bee$. Then $\tau_{m+1}\als<\tau_{m+1}\bes$ yields $k_m\als\leq k_m\bes$.
If $k_m\als=k_m\bes$, then $\alz<\bez$ and $\tau_m\al<\tau_m\be$.
If $k_m\als<k_m\bes$ then the  induction hypothesis yields $\tau_m k_m\als<\tau_m k_m\bes$ and $\om^{f(\ale)}\cdot \tau_m k_m\als+\alz<\om^{f(\ale)}\cdot \tau_m k_m\bes+\bez$, since  $\alz< \om^{f(\ale)}$. We then find that $\tau_m\al<\tau_m\be$.
So we may assume that $\ale<\bee$.


\medskip

\textit{Case 1: $k_m \als < \vt_m0$.} Then $\tau_m k_m \als = k_m \als$.
Hence,
\begin{align*}
\tau_m\al= {} & {} \Om_m\cdot \om^{\ale}+\om^{\om^\ale}(\om^{f(\ale)}\cdot k_m\als+\alz)+1\\
<  {} & {} \Om_m\cdot \om^{\bee}+\om^{\om^\bee}(\om^{f(\bee)}\cdot \tau_m k_m\bes+\bez)+1\\
= {}  & {} \tau_m \be
\end{align*}
follows in a straightforward way.



\medskip

\textit{Case 2: $k_m \als \geq \vt_m0$.} Using the definition, we then have $(\tau_m k_m\als)^-  +1 = \tau_m k_m \als$. We show that
\begin{align*}
\om^{\om^\ale}\om^{f(\ale)}\cdot (\tau_m k_m\als)^-  + \om^{\om^\ale}(\om^{f(\ale)}+\alz)+1
<  {}  {} ( \tau_m \be)^-
\end{align*}
holds, hence
\begin{align*}
\tau_m\al= {} & {} \Om_m\cdot \om^{\ale}+\om^{\om^\ale}\om^{f(\ale)}\cdot (\tau_m k_m\als)^-  + \om^{\om^\ale}(\om^{f(\ale)}+\alz)+1\\
<  {} & {} \Om_m\cdot \om^{\ale}+ ( \tau_m \be)^-\\
= {}  & {} (\tau_m \be)^-\\
< {}  & {} \tau_m \be.
\end{align*}

We know $\tau_m k_m \als < \tau_m \be$, hence
\[(\tau_m k_m\als)^- <(\tau_m\be)^-
=\Om_m\cdot \om^{\bee}+\om^{\om^\bee}(\om^{f(\bee)}\cdot \tau_m k_m\bes+\bez).\]
Therefore, $\om^{\om^{\ale}}\om^{f(\ale)}\cdot (\tau_m k_m\als)^- <  \om^{\om^{\ale}}\om^{f(\ale)}\cdot (\tau_m \be)^- = (\tau_m \be)^-$ because $\om^{\om^{\ale}}\om^{f(\ale)} = \om^{f(\om^\ale)}$ and $f(\om^\ale) < \om^{\ale+1} \leq \om^\bee$.


The last term in the normal form of $ \om^{\om^\bee}\cdot \bez$ is bigger than $\om^{\omega^{\be_1}}$.
Note that $\tau_{m+1}\bes=\Om_{m+1}\cdot \bee+\om^{f(\bee)}\cdot k_m\bes+\bez$. The observation just before this theorem yields $\bez >0$ otherwise $\bee$ is zero, a contradiction  (because $\bee> \ale$). So if
\[ \om^{\om^\ale}(\om^{f(\ale)}+\alz)+1 < \om^{\om^{\be_1}},\]
we can finish the proof by the standard observation
$\xi < \rho + \omega^\mu$ and $\lambda < \mu$ imply $\xi + \omega^\lambda < \rho  + \omega^\mu$.

Now,
\begin{align*}
{} & {} \om^{\om^\ale} (\om^{f(\ale)}+\alz)+1 \\
= {} & {} \om^{\om^\ale} \om^{f(\ale)}+ \om^{\om^\ale}  \alz+1 \\
< {} & {}   \om^{\om^\bee}
\end{align*}
because $\om^{\om^{\ale}} \alz  < \om^{\om^{\ale}}\om^{f(\ale)} = \om^{f(\om^\ale)}$ and $f(\om^\ale) < \om^{\ale+1} \leq \om^\bee$.

  \end{proof}

\begin{lemma}\label{boundsonTnall}
For all $\alpha \in T'_{n+1}[m+1]$ we have that if $\tau_{m+1} \alpha = \Omega_{m+1} \beta + \omega^{f(\beta)} k_{m} \alpha + \eta$, then
\[\begin{cases}
\beta < \omega^0 = \omega_0 & \mbox{ if } m\geq n,\\
\beta < \omega_{n-m} & \mbox{ if }m<n.
\end{cases}\]
\end{lemma}
\begin{proof}
We prove this by induction.
If $m\geq n$, then $\tau_{m+1} \alpha = \Omega_{m+1} 0 +  \omega^0 \alpha$, hence we are done. Assume $m<n$. If $\alpha = \vartheta_j \alpha'$ with $j<{m+1}$, then $\beta=0 < \omega_{n-m}$. Assume $\alpha = \vartheta_{m+1} \alpha'$. Assume
$\tau_{m+2} \alpha' = \Omega_{m+2} \beta' + \omega^{f(\beta')} k_{m+1} \alpha' + \eta'$ and
$\tau_{m+1} k_{m+1} \alpha' = \Omega_{m+1} \beta_1 + \omega^{f(\beta_1)} k_{m} k_{m+1} \alpha' + \eta_1$.
From the induction hypothesis, we know $\beta' < \omega_{n-m-1}$ and $\beta_1 < \omega_{n-m}$. Then
\begin{align*}
&\tau_{m+1} \alpha \\
={}& {}\Omega_{m+1} \omega^{\beta'} + \omega^{\omega^{\beta'}} (\omega^{f(\beta')}  ( \Omega_{m+1} \beta_1 + \omega^{f(\beta_1)} k_{m} k_{m+1} \alpha' + \eta_1)+ \eta')+1\\
={}& {}\Omega_{m+1} \omega^{\beta'} + \omega^{\omega^{\beta'}} \omega^{f(\beta')}  ( \Omega_{m+1} \beta_1 + \omega^{f(\beta_1)} k_{m}  \alpha' + \eta_1)+  \omega^{\omega^{\beta'}} \eta'+1\\
={}& {}\Omega_{m+1} (\omega^{\beta'}+ \beta_1)  + \omega^{\omega^{\beta'}} \omega^{f(\beta')}  ( \omega^{f(\beta_1)} k_{m}  \alpha' + \eta_1)+  \omega^{\omega^{\beta'}} \eta'+1.
\end{align*}
Now, $\omega^{\beta'} + \beta_1 < \omega_{n-m}$.
\end{proof}

\begin{lemma}\label{boundsonTnall[0]}
Let $n \geq 1$. For all $\alpha \in T'_{n+1}[0]$ we have that $\tau_{0} \alpha < \omega_{n+2}$.
\end{lemma}
\begin{proof} We prove this by induction on $lh(\alpha)$.
If $\alpha =0$, this is trivial. Assume $\alpha \in T'_{n+1}[0]$, meaning $\alpha = \vartheta_0 \alpha'$ with $\alpha' \in T'_{n+1}[1]$. Assume $\tau_1 \alpha' = \Omega_1 \beta' + \omega^{f(\beta')} k_0 \alpha' + \eta'$ with $\eta' < \omega^{f(\beta')}$. Using Lemma \ref{boundsonTnall}, we know that $\beta' < \omega_{n-0}  = \omega_n$.
Additionally, the induction hypothesis yields $\tau_0 k_0 \alpha' < \omega_{n+2}$. Now,
\begin{align*}
\tau_0 \vartheta_0 \alpha'  = \omega^{\omega^{\beta'}}(\omega^{f(\beta')} \tau_0 k_0 \alpha' + \eta')+1.
\end{align*}
From the definition of $f$, one obtains that $f(\beta') \leq \beta' \cdot \omega$. Hence, $\omega^{f(\beta')} \tau_0 k_0 \alpha' + \eta' < \omega^{f(\beta')} (\tau_0 k_0 \alpha' +1) < \omega_{n+2}$, so $\tau_0 \vt_0 \alpha' < \omega_{n+2}$.
\end{proof}

\begin{corollary}\label{otypeT'_n}
$otype(T'_{n+1}) \leq \omega_{n+2}$.
\end{corollary}
\begin{proof}
By Theorem \ref{tauiorderpreserving}, $\tau_0$ is an order preserving embedding from $T'_{n+1}[0] $ to $T^{all}_0 = E$. Furthermore, from Lemma \ref{boundsonTnall[0]}, we know $\tau_0 \alpha < \omega_{n+2}$ for all $\alpha \in T'_{n+1}[0]$. Hence $otype(T'_{n+1}) \leq \omega_{n+2}$.
\end{proof}

  \begin{corollary} $\vtn\vte\ldots\vt_n\Om_{n+1} \leq \om_{n+2}$.
\end{corollary}
 \begin{proof}
 By Lemma \ref{ordertypeT'_n}, we know
 \[\vtn\vte\ldots\vt_n\Om_{n+1} = otype(T_{n+1}[0]) = otype(T'_{n+1}[0]),\]
 hence the previous corollary yields $\vtn\vte\ldots\vt_n\Om_{n+1} \leq \om_{n+2}$.
\end{proof}

 \section{Binary $\vartheta$-functions}

So the question remains whether a maximal linear extension of $\overline{\S}^w_n$ can be realized by a suitable choice of \textit{unary} functions. It turns out that this, as we will show, is possible using specific \textit{binary} theta-functions. However, the question if this doable with unary functions remains open.
Let $n$ be a fixed non-negative integer. In this subsection, we also use the notation $T_n$, however it is different then the previous one.

\begin{definition} Let $T_n$ be the least set such that the following holds. On $T_n$, define $S$ and $K_i$.
\begin{enumerate}
\item
$0\in T_n$, $S0:=-1$, $K_i0:=\emptyset$,
\item if $\al,\be\in T_n$, $S\al\leq i+1$ and $S\be\leq i < n$, then $\oti\al\be\in T_n$, $S\oti\al\be:=i$ and
\[
K_j\oti\al\be:=\begin{cases}
K_j\al\cup K_j\be& \mbox{ if $j<i$,}\\
\{\oti\al\be\}& \mbox{ otherwise. }
\end{cases}\]
\end{enumerate}
\end{definition}

Note that all indices in $T_n$ are strictly smaller than $n$.

\begin{definition}
For $\oti\al\be,\oti\ga\de \in T_n$, define $\oti\al\be<\oti\ga\de$ iff either $i<j$ or $i=j$ and one of the following alternatives holds:
\begin{itemize}
\item $\al<\ga\und K_i\al\cup\{\be\}<\ot_j\ga\de$,
\item $\al=\ga\und \be<\de$,
\item $\al>\ga\und \oti\al\be\leq K_i\ga\cup \{\de\}$.
\end{itemize}
Let $0< \oti \al\be$ for all $\oti\al\be \in T_n \backslash \{0\}$.
\end{definition}

Here $\oti\al\be\leq K_i\ga\cup \{\de\}$ means that $\oti\al\be\leq \xi$ for every $\xi \in K_i\ga\cup \{\de\}$.

\begin{lemma}\label{basicpropertyoverlinetheta}
For $\oti \al \be \in T_n$, we have $\be < \oti \al \be$.
\end{lemma}
\begin{proof}
This can be proven by induction on $lh(\be)$.
\end{proof}

\begin{definition} Define $OT_n\subseteq T_n$ as follows.
\begin{enumerate}
\item
$0 \in OT_n$,
\item if $\al,\be\in OT_n$, $S\al\leq i+1$, $S\be\leq i<n$ and $K_i\al=\emptyset$, then $\oti\al\be\in OT_n$
\end{enumerate}
\end{definition}

Note that $K_i \al = \emptyset$ yields that $\al$ does not contain any $\ot_j$ for $j\leq i$.

\begin{definition} If $K_0 \al = \emptyset$, let $\al^-$ be the result of replacing every occurence of $\ot_i$ by $\ot_{i-1}$.
\end{definition}

\begin{lemma}\label{property minus}
If $\al<\be\und K_0\al=K_0 \be=\emptyset$, then $\al^-<\be^-$ and $(K_{i+1}\al)^-=K_i\al^-.$
\end{lemma}
\begin{proof}
This can be proven in a straightforward way by induction on $lh(\al)+lh(\be)$.
\end{proof}

\medskip

Therefore, if $\oti \al \be \in OT_n$, then $\al^-$ is defined and it is an element of $OT_{n-1}$. Additionally, if $i=0$, then $S(\al^-), S(\be) \leq 0$.

\begin{definition} Define $OT_n[0]$ as $OT_n \cap \Omega_1$, where $\Omega_1 := \ot_0 00 $
\end{definition}

\begin{definition}
Define $o_1 : OT_1[0]  \to \om$ as follows. An arbitrary element of $OT_1$ is of the form $\ot_0 (0 , \ot_0(0, \dots \ot_0(0,0)\dots))$. Define the image of this element under $o_1$ as $k$ if $\ot_0(\cdot,\cdot)$ occurs $k$ many times. Define $o_n:OT_n[0]\to \om_{2n-1}$ for $n>1$ as follows.
\begin{enumerate}
\item $o_n(0):=0$,
\item $o_n(\ot_0\al\be):=\varphi_{o_{n-1}(\al^-)}o_n(\be).$
\end{enumerate}
\end{definition}

Note that $S(\al^-),S(\be)\leq 0$ if $\ot_0 \al \be \in OT_n[0]$.

\begin{theorem} For every $n \geq 1$, $o_n$ is order-preserving and surjective.
\end{theorem}
\begin{proof}
The surjectivity of $o_n$ is easy to prove. We prove that $o_n$ is order-preserving. If $n=1$, this is trivial. Assume $n>1$ and assume that $o_{n-1}$ is order preserving. We will show that for all $\alpha, \beta \in OT_n[0]$, $\alpha < \beta$ yields $o_n(\alpha) < o_n(\beta)$. If $\alpha$ and/or $\beta$ are equal to zero, this is trivial. Assume $0 < \alpha < \beta$. Let $\alpha = \ot_0 \al_1 \al_2$ and $\be = \ot_0 \be_1 \be_2$. Then $\alpha < \beta$ iff one of the following cases holds:
\begin{enumerate}[noitemsep]
\item $\al_1 < \be_1$ and $\al_2 < \ot_0 \be_1 \be_2$,
\item $\al_1 = \be_1$ and $\al_2 < \be_2$,
\item $\al_1 > \be_1$ and $\ot_0 \al_1 \al_2 \leq \be_2$.
\end{enumerate}
Note that $\al_1 < \be_1$ yields $\al_1^- < \be_1^-$ by Lemma \ref{property minus}, hence $o_{n-1}(\al_1^-) < o_{n-1}(\be_1^-)$. Furthermore, the induction hypothesis yields that the previous case $i.$ is equivalent with the following case $i.$ for all $i$.
\begin{enumerate}[noitemsep]
\item $o_{n-1}\al_1^- < o_{n-1}\be_1^-$ and $o_n\al_2 < o_n\ot_0 \be_1 \be_2$,
\item $o_{n-1}\al_1^- = o_{n-1}\be_1^-$ and $o_n\al_2 < o_n\be_2$,
\item $o_{n-1}\al_1^- > o_{n-1}\be_1^-$ and $o_n\ot_0 \al_1 \al_2 \leq o_n\be_2$.
\end{enumerate}
Hence the above case $i.$ is equivalent with the following case $i.$:
\begin{enumerate}[noitemsep]
\item $o_{n-1}\al_1^- < o_{n-1}\be_1^-$ and $o_n\al_2 < \varphi_{o_{n-1}\be_1^-} o_n\be_2$,
\item $o_{n-1}\al_1^- = o_{n-1}\be_1^-$ and $o_n\al_2 < o_n\be_2$,
\item $o_{n-1}\al_1^- > o_{n-1}\be_1^-$ and $\varphi_{o_{n-1}\al_1^-} o_n\al_2 \leq o_n\be_2$.
\end{enumerate}
This is actually the definition of $\varphi_{o_{n-1}\al_1^-} o_n\al_2 < \varphi_{o_{n-1}\be_1^-} o_n\be_2$, so $o_n \ot_0 \al_1 \al_2 < o_n\ot_0 \be_1 \be_2$.

\end{proof}

This yields the following corollary.

\begin{corollary}\label{ordertypebinarytheta}
$otype(OT_n[0]) = \omega_{2n-1}$ if $n\geq 1$.
\end{corollary}

This ordinal notation system corresponds to a maximal linear extension of $\overline{\S}^s_n[0]=  \overline{\S}^w_n[0]$. 

\begin{definition} Define $f$ from $\overline{\S}^s_n$ to $OT_n$ as follows. $f(\varepsilon):=0$ if $\varepsilon$ is the empty sequence.
$f(ii_1\ldots i_k j\vec{s}):=\oti(f(i_1\ldots i_k))(f(j\vec{s}))$ if $i<i_1,\ldots,i_k$ and $j\leq  i$. This yields that $f(i)$ is defined as $\oti(0,0)$.
\end{definition}

\begin{lemma} $OT_n$ is a linear extension of $\overline{\S}^s_n$.
\end{lemma}
\begin{proof}
We prove by induction on the length of $s$ and $t$ that $ s \leq^s_{gap} t$ yields $f(s)\leq f(t)$.
If $s$ and/or $t$ are $\varepsilon$, then this is trivial. Assume not, then $s = i i_1 \dots i_k j \vec{s'}$  and $t = p p_1 \dots p_r q \vec{t'}$ with $i_1,\dots,i_k > i \geq j$ and $p_1,\dots,p_r > p \geq q$.
If $i<p$, then $f(s) \leq f(t)$ is trivial. Furthermore, $s \leq^s_{gap} t$ yields that $i>p$ is impossible. Therefore we can assume that $i=p$.
If the first $i$ of $s$ is mapped into $q\vec{t'}$ according to the inequality $s \leq^s_{gap} t$, then $i=q$ and $s \leq^s_{gap} \vec{qt'}$, hence $f(s) \leq f(qt')$. From Lemma \ref{basicpropertyoverlinetheta}, we know $f(q\vec{t'}) < f(t)$, hence we are done.
Assume that the first $i$ of $s$ is mapped onto the first $i=p$ of $t$ according to the $s \leq^s_{gap} t$ inequality. Then $j\vec{s'} \leq^s_{gap} q \vec{t'}$ and $i_1 \dots i_k \leq^s_{gap} p_1 \dots p_r$.
The induction hypothesis yields $f(j\vec{s'}) \leq f(q \vec{t'})$ and $f(i_1 \dots i_k) \leq f(p_1 \dots p_r)$. If $f(i_1 \dots i_k) = f(p_1 \dots p_r)$, then $f(s) \leq f(t)$ follows from $f(j\vec{s'}) \leq f(q \vec{t'})$. If $f(i_1 \dots i_k) < f(p_1 \dots p_r)$, then $f(s) \leq f(t)$ follows from $f(j\vec{s'}) \leq f(q \vec{t'}) < f(t)$ and $K_i(f(i_1 \dots i_k)) = \emptyset$.


\end{proof}

\begin{corollary} $OT_n[0]$ is a maximal linear extension of $\overline{\S}^w_n[0] = \overline{\S}^s_n[0]$.
\end{corollary}
\begin{proof}
The previous lemma yields that $OT_n[0]$ is a linear extension of $\overline{\S}_n[0]$. We also know that $otype(OT_n[0]) = \omega_{2n-1} = o(\overline{\S}_n[0])$.
\end{proof}

\medskip

In a sequel project, we intend to determine the relationship between other ordinal notation systems \textit{without} addition with the systems studied here. More specifically, we intend to look at ordinal diagrams \cite{Takeuti}, Gordeev-style ordinal notation systems \cite{GordeevAML1989} and non-iterated $\vartheta$-functions \cite{prooftheoryofimpredicativesubsystemsofanalysis,WilkenSimultanetheta}. This will be published elsewhere.


\renewcommand{\abstractname}{Acknowledgements}

\begin{abstract}
The first author's research was supported through  grant 
``Abstract Mathematics for Actual Computation: Hilbert's Program in the 21st Century" from the John Templeton Foundation. The opinions expressed in this publication are those of the authors and do not necessarily reflect the views of the John Templeton Foundation.

The second author wants to thank his funding organization Fellowship of the Research Foundation - Flanders (FWO).
\end{abstract}

\bibliography{Universal-gapsequences}
\bibliographystyle{abbrv}

\end{document}